\RequirePackage{etex}
\documentclass[graybox]{svmult}

\usepackage{type1cm}
\usepackage{microtype}

\usepackage{makeidx}         
\usepackage{graphicx}        
\usepackage{multicol}        
\usepackage[bottom]{footmisc}

\usepackage{nicefrac}
\usepackage{siunitx}
\usepackage{url}
\usepackage{booktabs}
\usepackage{multirow}

\usepackage{listings}
\usepackage[dvipsnames]{xcolor}
\definecolor{mygreen}{RGB}{28,172,0} 
\definecolor{mylilas}{RGB}{170,55,241}
\definecolor{matlabblue}{HTML}{0E07FF}
\definecolor{mygray}{rgb}{0.5,0.5,0.5}
\colorlet{codebg}{lightgray!20!white} %

\definecolor{brilliantrose}{rgb}{1.0, 0.33, 0.64}
\definecolor{myviolet}{rgb}{0.21, 0.0, 0.85}
\definecolor{amethyst}{rgb}{0.6, 0.4, 0.8}
\definecolor{carrotorange}{rgb}{0.93, 0.57, 0.13}

\usepackage[normalem]{ulem}

\usepackage{tikz,pgfplots}
\usetikzlibrary{matrix}
\usepackage{subfig}
\pgfplotsset{compat=1.18}
\usepackage[export]{adjustbox}
\usepackage{tikz-network}

\usepackage{todonotes}
\usepackage{comment}

\usepackage{newtxtext}       
\usepackage[varvw]{newtxmath}       

\usepackage[colorlinks]{hyperref}

\usepackage[linesnumbered,lined,boxed,commentsnumbered]{algorithm2e}

\makeindex

\begin{document}

\title*{Nearest Reversible Markov Chains with Sparsity Constraints: An Optimization Approach}

\titlerunning{Nearest Reversible Sparse Markov Chains}

\author{Stefano Cipolla\orcidID{0000-0002-8000-4719} and\\% 
Fabio Durastante\orcidID{0000-0002-1412-8289} and\\%
Miryam Gnazzo\orcidID{0000-0002-4381-5766} and\\%
Beatrice Meini\orcidID{0000-0003-3139-2513}}%

\authorrunning{S. Cipolla, F. Durastante, M. Gnazzo, B. Meini}

\institute{%
Stefano Cipolla \at School of Mathematical Sciences, University of Southampton, Building 54, Mathematical Sciences, Highfield, Southampton, SO17 1BJ, UK. \email{s.cipolla@soton.ac.uk}%
\and Fabio Durastante \at Dipartimento di Matematica, Università di Pisa, Largo Bruno Pontecorvo, 5 - 56127 Pisa, Italy  \email{fabio.durastante@unipi.it}%
\and Miryam Gnazzo \at Istituto di Scienza e Tecnologie dell'Informazione ``Alessandro Faedo'', CNR, Via G. Moruzzi, 1 - 56124 Pisa, Italy. \email{miryam.gnazzo@isti.cnr.it}%
\and Beatrice Meini \at Dipartimento di Matematica, Università di Pisa, Largo Bruno Pontecorvo, 5 - 56127 Pisa, Italy.\email{beatrice.meini@unipi.it}%
}

\maketitle

\abstract*{Reversibility is a key property of Markov chains, central to algorithms such as Metropolis--Hastings and other MCMC methods. Yet many applications yield non-reversible chains, motivating the problem of approximating them by reversible ones with minimal modification. We formulate this task as a matrix nearness problem and focus on the practically relevant case of sparse transition matrices. The resulting optimization problem is a quadratic programming problem, and numerical experiments illustrate the effectiveness of the approach. This framework provides a principled way to enforce reversibility and sparsity patterns in Markov chains with applications in MCMC, computational chemistry, and data-driven modeling.}
\abstract{Reversibility is a key property of Markov chains, central to algorithms such as Metropolis--Hastings and other MCMC methods. Yet many applications yield non-reversible chains, motivating the problem of approximating them by reversible ones with minimal modification. We formulate this task as a matrix nearness problem and focus on the practically relevant case of sparse transition matrices. The resulting optimization problem is a quadratic programming problem, and numerical experiments illustrate the effectiveness of the approach. This framework provides a principled way to enforce reversibility and sparsity patterns in Markov chains with applications in MCMC, computational chemistry, and data-driven modeling.}

\section{Introduction}

Markov chains~\cite{MR410929} are a fundamental tool in probability, statistics, and computational science. Their structural properties, i.e., the structural properties of their matrix representation, such as irreducibility, stationarity, or reversibility, play a role in both theoretical analysis and the construction of efficient algorithms. In particular, \emph{reversibility} is especially important since it forms the objective of widely used algorithms such as Metropolis--Hastings~\cite{10.1063/1.1699114,10.1093/biomet/57.1.97} and its more modern variants in Markov Chain Monte Carlo~\cite{MR1397966} (MCMC).

However, in many settings of interest, one starts from a given Markov chain that is \emph{not} reversible. This raises the natural question of how a given Markov chain can be modified as little as possible to yield a reversible one. Indeed, this question has been addressed in~\cite{MR3338930} for the case of a general irreducible Markov chain by means of a quadratic programming approach, and in~\cite{durastante2025riemannianoptimizationapproachfinding} via a Riemannian optimization approach~\cite{MR4533407} on a suitably defined manifold based on the multinomial manifold~\cite{8861409}. 

It is important to note that this problem is indeed motivated by applications in numerical linear algebra, where reversible chains often yield better spectral properties~\cite{Wolfer2021393}; in MCMC, where reversibility ensures convergence to the target distribution; and in data-driven modeling, where observed transition data may be noisy and fail to satisfy the reversibility condition~\cite{Trendelkamp}.

Formally, let $P \in \mathbb{R}^{n \times n}$ denote the transition matrix of a discrete, homogeneous Markov chain. The goal is to construct a reversible chain that is ``closest'' to $P$, in the sense of requiring the smallest possible perturbation to enforce reversibility. At the matrix level, this leads to a \emph{matrix nearness problem}~\cite{MR1041063}, of the form
\[
\mathrm{d}(P) = \min \left\{ \|X-P\| \;:\; X \in \mathbb{R}^{n \times n} \ \text{has property } \mathfrak{P} \right\}, %
\]
where $\mathfrak{P}$ denotes a prescribed property. The value $\mathrm{d}(P)$ then measures the minimum-norm perturbation needed to enforce the chosen property in the given matrix norm~$\|\cdot\|$. In the present setting, $\mathfrak{P}$ is the reversibility condition, expressed through the detailed balance equations
\[
\pi_i X_{ij} = \pi_j X_{ji}, \qquad \forall i,j,
\]
with respect to a stationary distribution $\boldsymbol{\pi}$, i.e., $\boldsymbol{\pi}^\top X = \boldsymbol{\pi}^\top$, $\boldsymbol{\pi}^\top \mathbf{1} = 1$, $\boldsymbol{\pi} > 0$,  where $X\ge 0$, $X\mathbf{1}=\mathbf{1}$, and the distance is measured in the Frobenius norm,
\[
\| A\|_F =\operatorname{trace}(A^\top A)^{1/2}.
\]

Previous work~\cite{MR3338930,durastante2025riemannianoptimizationapproachfinding} has focused primarily on the case where $P$ is irreducible, and any pair of states can be connected by a nonzero transition probability. In contrast, we are interested here in the setting where the transition matrix $P$ is \emph{sparse}. This is the typical situation in high-dimensional applications, where each state can only transition to a small number of states. In such cases, it is natural to require that the matrix $X$ obtained at the end preserves as much of the sparsity pattern of $P$ as possible, so that the resulting reversible chain remains computationally efficient to simulate and analyze. %
To ensure reproducibility, the code for all numerical experiments supporting the proposed approach is available at the GitHub repository \href{https://github.com/Cirdans-Home/sparse-reversible}{Cirdans-Home/sparse-reversible}.

The remainder of the paper is organized as follows. In Section~\ref{sec:notation} we introduce the notation and basic definitions that will be used throughout the work. Section~\ref{sec:problem_formulation} formally states the problem of finding the nearest reversible Markov chain, with sparse transition matrix, and discusses feasibility and convexity issues, with particular attention to the role of reversibility constraints. 
In Section~\ref{sec:IPM} we describe our proposed solution method based on a quadratic programming approach. Numerical experiments illustrating the performance and properties of the approach are presented in Section~\ref{sec:numerical_examples}. Finally, Section~\ref{sec:conclusions} summarizes the main findings and outlines directions for future research.

\subsection{Notation}\label{sec:notation}
We begin by establishing notation and recalling the necessary background; we refer the reader to \cite{MR410929} for a comprehensive introduction to Markov chains, and to \cite{MR1298430} for properties of nonnegative matrices.

A matrix $A\in\mathbb{R}^{m\times n}$ is said to be \emph{nonnegative} (\emph{positive}) if $A_{ij}\ge 0$ ($A_{ij}>0$) for any $i,j$, and will will write $A\ge 0$ ($A>0$).

\begin{definition}[Discrete-time homogeneous Markov chain]
A \emph{discrete-time homogeneous Markov chain} on a finite state space $\{1,\dots,n\}$ is a stochastic process $\{X_k\}_{k \geq 0}$ such that
\[
\mathbb{P}(X_{k+1}=j \mid X_k=i, X_{k-1},\dots,X_0) \;=\; \mathbb{P}(X_{k+1}=j \mid X_k=i) \;=\;P_{ij}, 
\]
for all $i,j \in \{1,\dots,n\}$. The transition probabilities $P_{ij}$ form the entries of a matrix $P=(P_{ij})$, called the \emph{transition matrix}. The homogeneity assumption means that $P$ does not depend on $k$.
\end{definition}

\begin{definition}[Stochastic matrices]
A \emph{stochastic} matrix $P\in \mathbb{R}^{n\times n}$ is a matrix such that
\[
P\mathbf{1}=\mathbf{1}, \; P \geq 0,
\]
where $\mathbf{1}= \left( 1,\ldots, 1 \right)^\top$.
The transition matrices of finite-state Markov chains are stochastic matrices.
\end{definition}

\begin{definition}[Stationary distribution]\label{def:stationary_distribution}
Let $P \in \mathbb{R}^{n\times n}$ be the transition matrix of a Markov chain on $\{1,\dots,n\}$.  
A probability vector $\boldsymbol{\pi}\in\mathbb{R}^n$ with $\boldsymbol{\pi} \geq 0$, $\boldsymbol{\pi}^\top \mathbf{1}=1$, is called a \emph{stationary distribution} for $P$ if 
\[
\boldsymbol{\pi}^\top P = \boldsymbol{\pi}^\top.
\]
If the chain is \emph{irreducible}, i.e., if for every pair of states $(i,j)$ there exists a $k\geq 0$ such that $[P^k]_{ij}>0$, then the Perron--Frobenius theorem \cite[\S 8, p.667]{MR1777382} implies that there exists a unique stationary distribution $\boldsymbol{\pi}$, the entries of $\boldsymbol{\pi}$ are strictly positive, i.e., $\pi_i > 0$ for all $i$, and that $\boldsymbol{\pi}$ is the unique probability vector that is a left eigenvector of $P$ associated with the eigenvalue $1$.
\end{definition}

\begin{definition}[Reversibility]\label{def:reversibility}
A Markov chain with irreducible transition matrix $P$ is said to be \emph{reversible with respect to a stationary distribution $\boldsymbol{\pi}$} if the \emph{detailed balance equations} hold:
\begin{equation}
\label{eq:detaied_balance}
\pi_i P_{ij} = \pi_j P_{ji}, \qquad \forall\, i,j.
\end{equation}
Equivalently, reversibility means that the probability flux between any two states is symmetric under $\boldsymbol{\pi}$.
\end{definition}

Consider a reversible Markov chain, with stationary distribution $\boldsymbol{\pi}$. Then, we may express the detailed balance equations in \eqref{eq:detaied_balance} in terms of the transition matrix $P$ of the Markov chain. Indeed, this translates into:
\[
D_{\boldsymbol{\pi}} P = P^{\top} D_{\boldsymbol{\pi}},
\]
where $D_{\boldsymbol{\pi}}= \operatorname{diag}\left(\pi_1,\ldots, \pi_n\right)$. Note that, by performing a similarity transformation with the diagonal matrix $D_{\boldsymbol{\hat\pi}}= \operatorname{diag}(\boldsymbol{\hat\pi})$, where $\hat{\boldsymbol{\pi}} = \boldsymbol{\pi}^{\nicefrac{1}{2}}$ entrywise, we get that
\[
D_{\boldsymbol{\pi}} P = P^\top D_{\boldsymbol{\pi}} \,\Leftrightarrow\, D_{\hat{\boldsymbol{\pi}}} P D_{\hat{\boldsymbol{\pi}}}^{-1} = D_{\hat{\boldsymbol{\pi}}}^{-1} P^\top D_{\hat{\boldsymbol{\pi}}} =  (D_{\hat{\boldsymbol{\pi}}} P D_{\hat{\boldsymbol{\pi}}}^{-1})^\top.
\]
Moreover, the stochasticity of $P$ leads to the additional property:
\[
D_{\hat{\boldsymbol{\pi}}} P D_{\hat{\boldsymbol{\pi}}}^{-1}\hat{\boldsymbol{\pi}} =  D_{\hat{\boldsymbol{\pi}}} P \mathbf{1} = D_{\hat{\boldsymbol{\pi}}} \mathbf{1} = \hat{\boldsymbol{\pi}}.
\]

The above motivates the introduction of the following set. Given a positive probability vector $\boldsymbol{\pi}$ we define
\begin{equation}
\label{def:symmetric_reversible}
    \mathcal{R}_{\boldsymbol{\pi}} \triangleq \{S \in \mathbb{R}^{n \times n} \;: \, S \geq 0, \; S=S^\top, \; S\hat{\boldsymbol{\pi}} = \hat{\boldsymbol{\pi}}, \; \hat{\boldsymbol{\pi}} = \boldsymbol{\pi}^{\nicefrac{1}{2}}\},
\end{equation}
which collects symmetric  non negative  matrices, with fixed eigenvector $\hat{\boldsymbol{\pi}}$ associated with eigenvalue $1$.

Since our main interest lies in Markov chains where each state can transition only to a limited number of other states---that is, chains whose transition matrix contains a number of nonzero entries proportional to $n$ rather than $n^2$---we introduce the following notation to describe and work with their sparsity pattern.

\begin{definition}[Sparsity pattern]
For any matrix $M \in \left\lbrace 0,1 \right\rbrace^{n \times n}$ we define its sparsity pattern as
\[
\Omega_M = \{(i,j) \;:\; 1 \leq i,j \leq n, \; M_{ij} \neq 0 \}, 
\qquad s_{M} = |\Omega_M|,
\]
and denote by
\[
\mathbb{S}(M) = \{ B \in \mathbb{R}^{n \times n} \;:\; B_{ij} \neq 0 \implies (i,j) \in \Omega_M \}
\]
the set of matrices that share the same sparsity pattern as $M$.
\end{definition}
In this work, we refer to Markov chains whose transition matrix $P$ is sparse as \emph{sparse Markov chains}.

{We end this section fixing a few additional notations, we denote by ``$\circ$'' the entry-wise or Hadamard product of two matrices. We denote by $\operatorname{triu}(\cdot)$ the operation of extracting the upper triangular part of a matrix.}

\section{Formulation of the optimization problem}\label{sec:problem_formulation}

Consider a given Markov chain with transition matrix $P \in \mathbb{R}^{n\times n}$, a positive probability vector $\boldsymbol{\pi}\in \mathbb{R}^n$ and a given irreducible matrix $M\in \left\lbrace 0,1\right\rbrace^{n\times n}$, which defines the sparsity pattern $\Omega_M$. We assume that $M$ is symmetric and $M_{ii}=1$ for any $i=1,\ldots,n$.

We are looking for a transition matrix $X \in \mathbb{R}^{n\times n}$, which is close to the original matrix $P$, and is reversible with respect to the probability distribution $\boldsymbol{\pi}$. Additionally, we require that the matrix $X$ has the prescribed sparsity pattern $\Omega_M$. Specifically, we seek a matrix $X \in \mathbb{R}^{n\times n}$ such that
\begin{equation}
\label{eq:initial const}
   X \mathbf{1} = \mathbf{1}, \; D_{\boldsymbol{\pi}} X= X^{\top} D_{\boldsymbol{\pi}},  \;X\geq 0, \; X\in \mathbb{S}(M).
\end{equation}
This translates into solving the optimization problem
\begin{equation} \label{eq:original-rever}
\begin{array}{rl}
         \displaystyle \min_{X \in \mathbb{R}^{n\times n} } & \displaystyle \frac{1}{2}\|X-P\|_F^2,  \\ %
        \text{s.t.} &  X \text{ satisfies the relations in \eqref{eq:initial const}}.
    \end{array}
\end{equation}
The term $\| X-P\|_F^2$ measures the distance between the original matrix $P$ and the final reversible one $X$. Consequently, this choice of the cost function drives the optimization procedure toward the reversible matrix closest to the original one, with respect to the distance induced by the Frobenius norm. {The choice of the standard Frobenius norm is motivated also by statistical considerations. The Frobenius norm represents the standard Euclidean distance in the vectorized space $\mathbb{R}^{n^2}$; minimizing the squared Frobenius norm is therefore equivalent to an element-wise least-squares criterion, which provides a maximum likelihood estimator, whenever we interpret $P$ as a noisy observation of $X$ and assume the noise being independent and identically distributed. 
We note, however, that a significant portion of our theoretical and algorithmic framework we discuss in the following can be extended to other distance measures, provided they retain convexity and regularity. For instance, incorporating a weighted Frobenius norm is straightforward and preserves both the quadratic structure and the efficient, closed-form projections. Extensions to the $\ell_1$ or $\ell_\infty$ norms also maintain convexity—allowing the problem to be solved via specialized splitting methods.}

As a first step, rather than working directly with the set of reversible stochastic matrices sharing the stationary distribution $\boldsymbol{\pi}$, we reformulate the problem in terms of the set $\mathcal{R}_{\boldsymbol{\pi}}$ defined in \eqref{def:symmetric_reversible}. As a consequence, we may perform the minimization in~\eqref{eq:original-rever} choosing as feasible set the matrices in $\mathcal{R}_{\boldsymbol{\pi}}$ with pattern $\Omega_M$.

This can be achieved performing the change of variable $Y=  D_{\hat{\boldsymbol{\pi}}}X D_{\hat{\boldsymbol{\pi}}}^{-1}$. Indeed, since $X$ is reversible with pattern $\Omega_M$, we conclude that $Y\in\mathcal{R}_{\boldsymbol{\pi}} \cap \mathbb{S}(M)$.

Then, employing this change of variables both in the objective function and in the feasible set, we construct the following optimization problem
in the variable $Y$:
\begin{equation} \label{eq:original}
\begin{array}{rl}
         \displaystyle \min_{Y \in \mathbb{R}^{n\times n} } & \displaystyle \mathrm{J}(Y) :=\frac{1}{2} \|D_{\hat{\boldsymbol{\pi}}}^{-1}Y D_{\hat{\boldsymbol{\pi}}} - P\|_F^2,  \\
        \text{s.t.} &  Y \in\mathcal{R}_{\boldsymbol{\pi}} \cap \mathbb{S}(M).
    \end{array}
\end{equation}

\begin{remark}\label{rmk:irreducible_and_reversible}
Observe that, by Definition~\ref{def:reversibility}, any reversible matrix $P$ with stationary vector $\boldsymbol{\pi}$ can be rescaled, via equivalence with the matrix $D_{\hat{\boldsymbol{\pi}}}$, so as to obtain an element of the set $\mathcal{R}_{\boldsymbol{\pi}}$ defined in~\eqref{def:symmetric_reversible}. 
Conversely, 
any matrix in the set $\mathcal{R}_{\boldsymbol{\pi}}$ in~\eqref{def:symmetric_reversible} is such that
the matrix $P=D_{\hat{\boldsymbol{\pi}}}^{-1} S  D_{\hat{\boldsymbol{\pi}}}$ is stochastic and satisfies 
the detailed balance relation 
\[
D_{\boldsymbol{\pi}} P = P^\top D_{\boldsymbol{\pi}}.
\]
However, $P$ might be reducible, since $S$ is not assumed to be irreducible, therefore $P$ 
is not in general reversible according to Definition~\ref{def:reversibility}. 

Observe that the set of matrices in $\mathcal{R}_{\boldsymbol{\pi}}$ that are irreducible is not a closed set. For example, consider the succession $R_{\varepsilon}\in \mathbb{R}^{n\times n}$ of stochastic reversible matrices in the~form
\[
R_{\varepsilon} = (1-\varepsilon)I + \frac{\varepsilon}{n} \mathbf{1} \mathbf{1}^T, \quad 0 < \varepsilon \leq 1,
\]
with probability distribution $\boldsymbol{\pi} = \nicefrac{\mathbf{1}}{n}$. For each $0 < \varepsilon \leq 1,$ the matrix $R_{\varepsilon}$ is irreducible. Observe that $R_{\varepsilon} \rightarrow I$, as $\varepsilon\rightarrow 0$. While both $R_{\varepsilon}$ and $I$ belong to the set $\mathcal{R}_{\boldsymbol{\pi}}$, we have that the limit $I$ is not irreducible. Therefore, to have a feasible closed set, the matrices in the set  $\mathcal{R}_{\boldsymbol{\pi}}$ are not required to be irreducible.
\end{remark}

Before solving the optimization problem in~\eqref{eq:original}, we prove that, %
given a probability vector $\boldsymbol{\pi}>0$ and an irreducible and symmetric sparse matrix $M\in \left\lbrace 0,1 \right\rbrace^{n\times n}$, such that $M_{ii}=1$ for any $i=1,\ldots,n$, the feasible set
\[
\left\lbrace Y \in \mathbb{R}^{n \times n}: Y \in\mathcal{R}_{\boldsymbol{\pi}}, \; Y \in \mathbb{S}(M)  \right\rbrace
\]
is convex and non empty. 

{
\begin{remark}\label{rmk:miqp_formulation}
We observe that in principle we could try to optimize also for the position of the nonzero entries by modifying the constraints, e.g., by constraining the number of nonzero entries per row of the pattern. Recalling that, given a maximum row-degree vector $\mathbf{d} \in \mathbb{N}^n$, this problem can be formulated as
\[
\begin{split}
\min_{X,M} & \frac{1}{2}\| X - P \|_F^2, \\
& \begin{array}{ll}
\text{s.t.}& X \mathbf{1} = \mathbf{1},\\
& D_{\boldsymbol{\pi}} X = X^\top D_{\boldsymbol{\pi}},\\
& 0 \leq X \leq M, \quad M \mathbf{1} \leq \mathbf{d}, \quad M \in \{0,1\}^{n \times n},\\
& M = M^\top, \\
\end{array}
\end{split}
\]   
this simultaneous optimization of both the transition probabilities $X$ and the topology $M$ gives a Mixed-Integer Quadratic Program. Due to the combinatorial nature of the binary constraints on $M$, the problem becomes non-convex and NP-hard~\cite{MR3612939}, which contrasts with the tractability of the fixed-sparsity variant we consider here.
\end{remark}}

\subsection{Constraint feasibility and convexity}\label{sec:feasible_and_convex}

{Suppose we are given a probability distribution $\boldsymbol{\pi}>0$ and a stochastic matrix $Q$ possibly not satisfying $\boldsymbol{\pi}^\top Q = \boldsymbol{\pi}^\top$. We wish to construct a stochastic matrix $T$ such that $D_{\boldsymbol{\pi}} T = T^{\top} D_{\boldsymbol{\pi}}$, where $T$ is obtained by modifying only the nonzero and the diagonal entries of $Q$. A standard approach~\cite{Choi_Wolfer} is to construct a \emph{proposal transition matrix} obtained by correcting the dynamics of $Q$ by introducing an acceptance mechanism,}
yielding a new transition matrix 
$T = (T_{ij})_{i,j \in \{1,\ldots,n\}}$ defined by
\[
T_{ij} \;=\; Q_{ij} \, \alpha_{ij}, \qquad i \neq j,
\]
where $\alpha_{ij} \in [0,1]$, for $i,j=1,\ldots,n$ is an acceptance probability. Finally, the diagonal entries are adjusted so that each row sums to one:
\[
T_{ii} \;=\; 1 - \sum_{j \neq i} T_{ij}.
\]
Observe that, since $\alpha_{ij}\in[0,1]$ and $Q$ is a stochastic matrix, we have that for $i \neq j$
\[
T_{ij} = Q_{ij}\,\alpha_{ij} \leq Q_{ij},
\]
hence
\[
\sum_{j \neq i} T_{ij} \;\leq\; \sum_{j \neq i} Q_{ij} \;=\; 1 - Q_{ii}.
\]
Therefore,
\[
T_{ii} \;=\; 1 - \sum_{j \neq i} T_{ij}
\;\;\geq\;\; 1 - (1 - Q_{ii})
\;=\; Q_{ii} \;\geq\; 0.
\]
Thus every diagonal entry $T_{ii}$ is nonnegative, and the acceptance mechanism preserves the nonnegativity of the matrix. The acceptance rule $\alpha_{ij}$ is then chosen to enforce the detailed balance condition
\begin{equation}\label{eq:reversibilization}
    \pi_i Q_{ij} \alpha_{ij} \;=\; \pi_j Q_{ji} \alpha_{ji}, \qquad \forall i,j \in \{1,\ldots,n\}.
\end{equation}
By construction, $T$ is reversible with respect to $\boldsymbol{\pi}$. The respect of the detailed balance condition and the row stochasticity of $T$ imply that the constructed matrix $T$ has as stationary distribution the vector $\boldsymbol{\pi}$. As final step, we observe that the sparsity pattern of $Q$ is directly inherited by $T$ plus, possibly, nonzero diagonal entries. Indeed, if $Q_{ij}=0$, then necessarily $T_{ij}=0$, since no transition from $i$ to $j$ can be proposed in the first place. Moreover, the diagonal entries might be nonzero to enforce the stochasticity.

Thus, the choice of $Q$ determines the graph of possible moves, while the acceptance rule only re-scales the allowed transitions and adds at most entries on the diagonal. %

Two classical ways of choosing the acceptance probability $\alpha_{ij}$ are:
\begin{description}[Metropolis]
    \item[\textbf{Metropolis rule (Metropolis--Hastings algorithm)}~\cite{10.1063/1.1699114,10.1093/biomet/57.1.97}:]  
    The acceptance probability is taken to be
    \begin{equation}\label{eq:reversibilization_metropolis}
    \alpha_{ij} \;=\; \min\!\left(1, \; \frac{\pi_j Q_{ji}}{\pi_i Q_{ij}} \right).
    \end{equation}
    \item[\textbf{Barker's algorithm}~\cite{Barker1965}:]
    An alternative acceptance function is
    \[
    \alpha_{ij} \;=\; \frac{\pi_j Q_{ji}}{\pi_i Q_{ij} + \pi_j Q_{ji}}.
    \]
\end{description}
Both constructions enforce reversibility of $T$ with respect to the target distribution $\boldsymbol{\pi}$, in general these are not the only closed-form constructions possible, see, e.g.,~\cite{Choi_Wolfer}.

The constructions above show how one may always build a reversible Markov chain with stationary distribution $\boldsymbol{\pi}$, starting from an arbitrary proposal kernel $Q$. We can also characterize when a given transition kernel $Q$ itself is already reversible with respect to $\boldsymbol{\pi}$ by means of the \emph{Kolmogorov cycle condition}~\cite[Theorem 1.7]{MR554920}. For a Markov chain with transition matrix $T=(T_{ij})$ and stationary distribution $\boldsymbol{\pi}$, reversibility holds if and only if for every finite cycle of distinct states
\[
i_1 \;\to\; i_2 \;\to\; \cdots \;\to\; i_k \;\to\; i_1,
\]
the following equality is satisfied:
\[
T_{i_1 i_2} \, T_{i_2 i_3} \, \cdots \, T_{i_{k-1} i_k} \, T_{i_k i_1}
\;=\;
T_{i_1 i_k} \, T_{i_k i_{k-1}} \, \cdots \, T_{i_3 i_2} \, T_{i_2 i_1}.
\]
Equivalently, one can state the condition in terms of the stationary distribution:
\[
\pi_{i_1} T_{i_1 i_2} \, \pi_{i_2} T_{i_2 i_3} \cdots \pi_{i_k} T_{i_k i_1}
\;=\;
\pi_{i_1} T_{i_1 i_k} \, \pi_{i_k} T_{i_k i_{k-1}} \cdots \pi_{i_2} T_{i_2 i_1}.
\]
When the cycle condition holds, detailed balance is satisfied, and the chain is reversible with respect to $\boldsymbol{\pi}$. If the cycle condition fails for even one cycle, the chain is inherently non-reversible, and one can resort to a construction such as Metropolis--Hastings, Barker or other systematic approaches~\cite{Choi_Wolfer} to enforce reversibility. We stress that the Kolmogorov cycle condition depends only on the transition structure of $T$, hence if $T$ is non-reversible, it \emph{cannot be modified locally} to restore reversibility without changing its transition graph. In contrast, by introducing a proposal kernel $Q$ and an acceptance rule, one can always build a new reversible chain, though the resulting dynamics may differ significantly from those of the original $Q$ and may have added some diagonal entries to the transition matrix.

\begin{remark}\label{rmk:irreducible}
When the Metropolis--Hastings, or equivalently Barker's, procedure is applied to a {transition matrix} with a
non-symmetric sparsity pattern, the resulting transition matrix may have a
different support than the original one. In particular, if $Q_{ij} > 0$ but
$Q_{ji} = 0$, then the detailed balance condition cannot be satisfied, and the
Metropolis--Hastings construction sets both $T_{ij}$ and $T_{ji}$ equal to
zero. As a consequence, all non-reciprocated transitions are eliminated,
leaving only those edges of the graph that appear in both directions. In
graph-theoretic terms, this means that the transition structure after the
Metropolis--Hastings adjustment corresponds to the largest symmetric subgraph
contained in the original directed graph. Consider, e.g., the stochastic $4 \times 4$  matrix
\[
Q = \begin{bmatrix}
0   & 1   & 0   & 0   \\
0   & 0   & 1   & 0   \\
\nicefrac{1}{2} & 0   & 0   & \nicefrac{1}{2} \\
0   & 0   & 1   & 0
\end{bmatrix}.
\]
The corresponding directed graph is shown
in Fig.~\ref{fig:mh_symmetrization_a}. Applying the Metropolis--Hastings
adjustment eliminates the one-way transitions $1 \to 2$ and $2 \to 3$, since
their reciprocals do not exist, but preserves the bidirectional link between
states $3$ and $4$. The resulting adjusted transition matrix has the symmetric
support
\[
{T} = \begin{bmatrix}
1 & 0 & 0 & 0 \\
0 & 1 & 0 & 0 \\
0 & 0 & \nicefrac{1}{2} & \nicefrac{1}{2} \\
0 & 0 & 1 & 0
\end{bmatrix},
\]
and its associated graph is shown in Fig.~\ref{fig:mh_symmetrization_b}. This has the effect of also changing the structure of the associated Markov chain, which can then become reducible as in the example. 
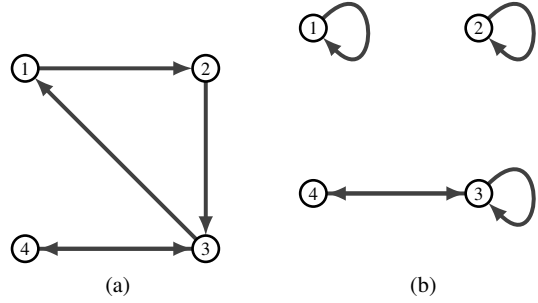
\begin{figure}[htpb]
\sidecaption[t]
\subfloat[\label{fig:mh_symmetrization_a}]{%
\begin{tikzpicture}[scale=1.2]
\Vertex[IdAsLabel=true,size=0.35,color=white]{1}
\Vertex[x=2,y=0,IdAsLabel=true,size=0.35,color=white]{2}
\Vertex[x=2,y=-2,IdAsLabel=true,size=0.35,color=white]{3}
\Vertex[x=0,y=-2,IdAsLabel=true,size=0.35,color=white]{4}
\Edge[Direct](1)(2)
\Edge[Direct](2)(3)
\Edge[Direct](3)(1)
\Edge[Direct](3)(4)
\Edge[Direct](4)(3)
\end{tikzpicture}
}
\hspace{3em}
\subfloat[\label{fig:mh_symmetrization_b}]{%
\begin{tikzpicture}[scale=1.1]
\Vertex[IdAsLabel=true,size=0.35,color=white]{1}
\Vertex[x=2,y=0,IdAsLabel=true,size=0.35,color=white]{2}
\Vertex[x=2,y=-2,IdAsLabel=true,size=0.35,color=white]{3}
\Vertex[x=0,y=-2,IdAsLabel=true,size=0.35,color=white]{4}
\Edge[Direct](1)(1)
\Edge[Direct](2)(2)
\Edge[Direct](3)(3)
\Edge[Direct](3)(4)
\Edge[Direct](4)(3)
\end{tikzpicture}}
\caption{Effect of the Metropolis--Hastings adjustment on a non-symmetric
stochastic matrix. The Markov chain represented by the graph in panel~\ref{fig:mh_symmetrization_a} is irreducible, i.e., the underlying directed graph is strongly connected. On the other hand, its \emph{reversibilization} obtained through the Metropolis--Hastings approach in panel~\ref{fig:mh_symmetrization_b} is reducible, and the corresponding graph is made of three connected components.}
\label{fig:mh_symmetrization}
\end{figure}

\end{remark}

Following the idea in \cite[Proposition 2]{durastante2025kemenysconstantminimizationreversible} and the previous insights, we may prove that the {feasible} set  is not empty.

%

%
%

%
%

\begin{comment}
\begin{proposition}\label{pro:it_is_not_empty}
Let $A \in \mathbb{R}^{n \times n}$ with $A \geq 0$, $A \mathbf{1} = \mathbf{1}$, $\boldsymbol{\pi}^\top A = \boldsymbol{\pi}^\top$ with $\boldsymbol{\pi} > 0$, and $\boldsymbol{\pi}^\top \mathbf{1} = 1$, then the set of matrices $\Delta \in \mathbb{R}^{n \times n}$ with $\Delta \in \mathbb{S}(A+A^\top +I)$ and satisfying
\[
\Delta \mathbf{1} = \mathbf{0}, \qquad D_{\boldsymbol{\pi}}(A + \Delta) = (A + \Delta)^\top D_{\boldsymbol{\pi}}, \qquad A + \Delta \geq 0,
\]
is not empty.
\end{proposition}

\begin{proof}
By either the Metropolis--Hastings or the Barker algorithms we know that there exist stochastic kernels $Q$, and weighting matrix $\alpha$ such that the detailed balance condition for $P = Q \alpha$ hold, i.e., for which we have $P \mathbf{1} = \mathbf{1}, \text{ and } D_{\boldsymbol{\pi}}P = P^\top D_{\boldsymbol{\pi}}$. We then need to show that we can select $Q$ in such a way that
\[
A + \Delta = P = Q \alpha, \quad \Delta \mathbf{1} = \mathbf{0}, \text{ and}\quad \Delta \geq -A. 
\]
Indeed, for whatever choice of $Q$ we find \[\Delta \mathbf{1} = (P - A)\mathbf{1} = P\mathbf{1} - A\mathbf{1} = \mathbf{1} - \mathbf{1} = \mathbf{0}.\]
Hence the problem is reduced to the choice of a $Q$ such that $Q\alpha - A \geq - A$, i.e., $Q \alpha \geq 0$ which is guaranteed for any choice of $Q$.
\end{proof}
\end{comment}

%
%
%
%
%
%
%

%
%
%

\begin{proposition}
\label{prop:it_is_not_empty}
Consider a probability distribution $\boldsymbol{\pi}>0$, and a pattern {induced by} a given irreducible matrix $M\in \left\lbrace 0,1 \right\rbrace^{n\times n}$. Assume that $M$ is symmetric and $M_{ii}=1$ for any $i=1,\ldots,n$. Then, the set of {irreducible} matrices $Y$ such that $Y \in\mathcal{R}_{\boldsymbol{\pi}}$ as in~\eqref{def:symmetric_reversible} and $Y \in \mathbb{S}(M)$ is not empty.
\end{proposition}

\begin{proof}
{Consider the matrix $Q=D_{M\mathbf{1}}^{-1} M$. Observe that $Q$ is an irreducible stochastic matrix in ${\mathbb{S}(M)}$ since $M$ is irreducible.}
By either the Metropolis--Hastings or the Barker algorithms, we know that there exist 
a
weighting matrix $\alpha$ such that the detailed balance conditions for $S = Q \circ \alpha$ hold, i.e., for which we have $S\mathbf{1} = \mathbf{1}, \text{ and } D_{\boldsymbol{\pi}}S = S^\top D_{\boldsymbol{\pi}}$. In addition, we have $S\in \mathbb{S}(M)$. 
{Moreover, since $M$ is symmetric, in view of Remark~\ref{rmk:irreducible}, $\alpha_{ij}\alpha_{ji}\ne 0$ for any $(i,j)\in\Omega_M$, $i\ne j$, therefore $S$ is irreducible as well.} Scaling by the diagonal matrix $D_{\hat{\boldsymbol{\pi}}}$ from both sides, we get that $D_{\hat{\boldsymbol{\pi}}} S D_{\hat{\boldsymbol{\pi}}}^{-1} \in\mathcal{R}_{\boldsymbol{\pi}}$, and this scaling preserves the pattern $\mathbb{S}(M)$. Therefore, we can choose the matrix $Y$ as $D_{\boldsymbol{\hat\pi}} S D_{\boldsymbol{\hat\pi}}^{-1}$. 
\end{proof}

\begin{proposition}\label{pro:convexity}
Consider a probability distribution $\boldsymbol{\pi}>0$, and a pattern taken from a given irreducible matrix $M\in \left\lbrace 0,1\right\rbrace^{n\times n}$. Assume that $M$ is symmetric and $M_{ii}=1$ for any $i=1,\ldots,n$. Then, the intersection $\mathcal{R}_{\boldsymbol{\pi}} \cap \mathbb{S}(M)$ is a convex set.
\end{proposition}

\begin{proof}
Let $Y_1$ and $Y_2$ be two matrices in $\mathcal{R}_{\boldsymbol{\pi}} \cap \mathbb{S}(M)$. By the definition of set $\mathcal{R}_{\boldsymbol{\pi}}$ in~\eqref{def:symmetric_reversible}, this implies that $Y_i$ is nonnegative, symmetric, and preserves $\hat{\boldsymbol{\pi}}$, i.e., for $i=1,2$
\[
Y_i \ge 0, \, Y_i = Y_i^{\top}, \; \mbox{and } Y_i \hat{\boldsymbol{\pi}} = \hat{\boldsymbol{\pi}}.
\]
 Consider a convex combination ${Y}_\lambda = \lambda {Y}_1 + (1-\lambda){Y}_2$ with $\lambda \in [0,1]$. $Y_{\lambda}$ is clearly nonnegative since the set of nonnegative matrices is convex. Moreover, the symmetry is preserved, since
 \[
 ({Y}_\lambda)^\top = \lambda {Y}_1^\top + (1-\lambda){Y}_2^\top = {Y}_\lambda.
 \]
Finally, ${Y}_\lambda \hat{\boldsymbol{\pi}} = \lambda {Y}_1\hat{\boldsymbol{\pi}} + (1-\lambda)Y_2\hat{\boldsymbol{\pi}} = \hat{\boldsymbol{\pi}}$, showing that $\boldsymbol{\hat\pi}$ is an eigenvector associated with $1$ for the convex combination $Y_{\lambda}$. Therefore, $Y_\lambda \in\mathcal{R}_{\boldsymbol{\pi}}$ for all $\lambda \in [0,1]$. Finally, we observe that the pattern is preserved, since $Y_1,Y_2 \in \mathbb{S}(M)$ implies ${Y}_{\lambda} \in \mathbb{S}(M)$.
We then conclude that the feasible set in the optimization problem~\eqref{eq:original} is convex.
\end{proof}

This implies that the matrix nearness problem with respect to Frobenius norm $\|\cdot\|_{F}$ admits a unique solution, as it reduces to the minimization of a strictly convex quadratic objective function
over a convex feasible set. This is consistent with the behavior observed in the unstructured case studied in~\cite{durastante2025riemannianoptimizationapproachfinding,MR3338930}. 

Furthermore, we remark that the direct use of the Metropolis--Hastings approach, or similar ones, does not allow to control any norm of the %
perturbation applied to the starting matrix, since it only performs a reversibilization of a given Markov chain, without any theoretical guarantees on the closeness from the original one; see Fig.~\ref{fig:norm_distance_to_metropolis} in Section~\ref{sec:numerical_examples} for an example of this.

\section{Reformulation of the optimization problem}\label{sec:IPM}

A first reformulation of problem \eqref{eq:original} is obtained by vectorizing both the objective function and the constraints. 
{To express the vectorization we use the $\operatorname{vec}$ operator, which transforms a matrix into a vector by stacking its columns sequentially. Specifically, given a matrix $A = (A_{ij}) \in \mathbb{R}^{m \times n}$, the vector $\mathbf{a} = \operatorname{vec}(A)$ has entries $a_k = A_{ij}$ with $k=i + m(j-1)$.
A key identity involving this operator is $\operatorname{vec}(AXB) = (B^\top \otimes A)\,\operatorname{vec}(X)$, where ``$\otimes$'' denotes the Kronecker product and $A$, $X$, and $B$ are matrices of compatible dimensions (with $A$ or $B$ possibly being vectors).}

In the following lemma, we provide the main constructive steps behind this {reformulation}.
\begin{lemma}
\label{lem:vectorization-on-the-constriant}
    Consider a probability vector $\boldsymbol{\pi}>0$, and a sparsity pattern $\Omega_M$, where $M\in \left\lbrace 0,1 \right\rbrace^{n\times n}$ is irreducible, symmetric and such that $M_{ii}=1$ for any $i=1,\ldots,n$. %
    Then, the set of constraints in~\eqref{eq:original} for the matrix $Y \in \mathbb{R}^{n\times n}$ can be rephrased as a set of constraints on $\operatorname{vec}(Y)$ as follows:
    \begin{equation}
    \label{eq:constriant-vec}
    \begin{array}{rl}
        &({{\hat{\boldsymbol{\pi}}}^{\top}} \otimes I)\operatorname{vec}(Y) = \hat{\boldsymbol{\pi}}, \\
        & ( I_{n^2} - K) \operatorname{vec}(Y) = 0, \\ 
        & \operatorname{diag}(\operatorname{vec}(\mathbf{1}\mathbf{1}^{\top} - M \circ \mathbf{1}\mathbf{1}^{\top}))\operatorname{vec}(Y)= 0,   \\
        & \operatorname{vec}(Y) \geq 0,
    \end{array}
    \end{equation}
where $K$ is the orthogonal commutation matrix such that $K\operatorname{vec}(\Delta)=\operatorname{vec}(\Delta^\top)$.
\end{lemma}

\begin{proof}
The first constraint in~\eqref{eq:constriant-vec} arises form the eigenvector relation $Y \hat{\boldsymbol{\pi}} = \hat{\boldsymbol{\pi}}$, using the relation
\[
\operatorname{vec}(Y \hat{\boldsymbol{\pi}}) = (\hat{\boldsymbol{\pi}}^{\top} \otimes I) \operatorname{vec}(Y).
\]
The second constraint arises from the relation $Y^{\top}= Y$; indeed, vectorizing we get
\[
(I \otimes I - (I \otimes I) K) \operatorname{vec}(Y) =  (I_{n^2} -  K) \operatorname{vec}(Y)=0,
\]
where we employed the relation $\operatorname{vec}(Y) = (I \otimes I)\operatorname{vec}(Y)$ and $\operatorname{vec}(Y^\top)= K\operatorname{vec}(Y)$.

The third constraint models the relation associated with the pattern $\Omega_M$. In practice, we are imposing that for the entries $(i,j) \not \in \Omega_M$, the corresponding element $Y_{ij} =0$. This translates as a constraint on $\operatorname{vec}(Y)$, imposing that
\[
\operatorname{diag}(\operatorname{vec}(\mathbf{1}\mathbf{1}^{\top} - M \circ \mathbf{1}\mathbf{1}^{\top}))\operatorname{vec}(Y)= 0.
\]
Lastly, the constraint on the nonnegativity of $Y$ can be obtained performing a vectorization~step.
\end{proof}

As a result of the above discussion, vectorizing the objective function and taking into account all the constraints in Lemma~\ref{lem:vectorization-on-the-constriant}, we obtain from the optimization problem in~\eqref{eq:original}:
\begin{equation}\label{eq:p1}
        \begin{array}{rl}
         \displaystyle \min_{Y \in\mathbb{R}^{n \times n}} & \displaystyle \mathrm{J}(Y) = \frac{1}{2}\|(D_{\hat{\boldsymbol{\pi}}} \otimes D_{\hat{\boldsymbol{\pi}}}^{-1})\operatorname{vec}(Y) - \operatorname{vec}(P)\|_2^2,  \\
        \text{s.t.} & 
        (\hat{\boldsymbol{{\boldsymbol{\pi}}}}^{\top} \otimes I)\operatorname{vec}(Y) = \hat{\boldsymbol{\pi}}, \\
        & (I_{n^2} -  K   ) \operatorname{vec}(Y) = 0, \\ 
        & \operatorname{diag}(\operatorname{vec}(\mathbf{1}\mathbf{1}^{\top} - M \circ \mathbf{1}\mathbf{1}^{\top}))\operatorname{vec}(Y)= 0,   \\
        & \operatorname{vec}(Y) \geq 0.
    \end{array}  
\end{equation}

Recall that $M$ is assumed to be symmetric and to contain the diagonal, and let $s_M$ denote the cardinality of the pattern $\mathbb{S}(M)$. By exploiting the symmetric structure of $\mathcal{R}_{\boldsymbol{\pi}}$, problem~\eqref{eq:original} admits a more compact reformulation than~\eqref{eq:p1}, in terms of both the number of variables and constraints. To this end, we observe that any matrix $Y \in \mathbb{S}(M)$ satisfying the symmetry constraint $Y = Y^\top$ can be expressed as
\begin{equation}\label{eq:constrained_matrix}
Y = M \circ
\begin{bmatrix}
        Y_{11} & Y_{12} & \cdots & \cdots & Y_{1n} \\
        Y_{12} & \ddots & \ddots & & \vdots \\
        \vdots & \ddots  & \ddots & \ddots   &               \vdots   \\
           \vdots   &   &        \ddots    &   \ddots    &  Y_{(n-1)n}\\
        Y_{1n} & \cdots & \cdots & Y_{(n-1)n} & Y_{nn}
\end{bmatrix}.
\end{equation}
In other words, the entries in the strictly lower triangular part are uniquely determined by the nonzero ones in the upper triangular part.

Hence, we may reformulate the optimization problem in terms of only those variables allowed to be nonzero, corresponding to the sparsity pattern of the upper triangular part~$\mathbb{S}(M)$.

To this aim, let $y_M = (s_M - n)/2 + n$, with $n \leq y_M \leq s_M$,  and define the vector
\[
\mathbf{y} = \begin{bmatrix}
Y_{i_1 j_1} & Y_{i_2 j_2} & \cdots & Y_{i_{y_M} j_{y_M}}
\end{bmatrix}^\top,
\]
where each pair $(i_k, j_k) \in \Omega_{\operatorname{triu}(M)}$.

We start defining an operator mapping the vector $\mathbf{y}$ into the upper triangular part of $Y$. To this end, define the matrix constructor $S(\mathbf{y}) \in \mathbb{R}^{n \times n}$ defined elementwise as
\begin{equation}\label{eq:S_matrix_operator}
    (S(\mathbf{y}))_{ij} =
\begin{cases}
Y_{ij}, & (i,j) \in  \Omega_{\operatorname{triu}(M)},\ i < j, \\[6pt]
\frac{1}{2}Y_{ii}, & i = j,  \\[6pt]
{0}, &\hbox{otherwise}.
\end{cases}
\end{equation}
We have, hence, $Y = S(\mathbf{y}) + S(\mathbf{y})^\top  $ and $\operatorname{vec}({Y})=\operatorname{vec}({S(\mathbf{y})})+ K\operatorname{vec}({S(\mathbf{y})})$. Let us now define by $r(i,j) = (j-1)n + i$ the row index corresponding to entry $(i,j)$ in the vectorization in $\operatorname{vec}({S(\mathbf{y})})$ and by $k(i,j)$ the index of $Y_{ij}$ in $\mathbf{y}$ whenever $(i,j) \in  \Omega_{\operatorname{triu}(M)}$. With these definitions, we can consider the operator $\Pi \in {\{0,1\}}^{n^2 \times y_M}$ as

    \begin{equation}
    \Pi_{\,r(i,j),\,k} =
    \begin{cases}
    1, & (i,j) \in \Omega_{\operatorname{triu}(M)},\ i < j,\ k = k(i,j), \\[6pt]
     \frac{1}{2}, & i = j,\ k = k(i,j), \\[6pt]
    0, & \text{otherwise},
    \end{cases}\label{eq:pi_projector}
    \end{equation}
i.e., $\Pi \mathbf{y} = \operatorname{vec}(S(\mathbf{y}))$. Moreover, defining the weighting vector $\mathbf{d} \in \mathbb{R}^{y_M}$ with  
\[
    d_{k} =
    \begin{cases}
    1, & (i,j) \in \Omega_{\operatorname{triu}(M)},\ i < j,\ k = k(i,j), \\[6pt]
     \frac{1}{2}, & i = j,\ k = k(i,j), \\[6pt]
    \end{cases}
    \]
we have

\[
\Pi^\top  ( \Pi \mathbf{y} + K\Pi \mathbf{y}  ) =\Pi^\top  (\operatorname{vec}{(S(\mathbf{y}))} + \operatorname{vec}{(S(\mathbf{y})^\top )}  ) = \operatorname{diag}(\mathbf{d})\mathbf{y},
\]
where we used $\Pi^\top  \operatorname{vec}{(L)} = 0$ for any strictly lower triangular matrix $L$.  Using the observations carried out until now, we find $$\operatorname{vec}({Y}) = (I+K)\Pi\mathbf{y}.$$ Defining now, for the sake of brevity $\mathbf{p}=\operatorname{vec}(P)$, we write \eqref{eq:p1}  as
\begin{equation}\label{eq:p1_compact_final}
        \begin{array}{rl}
      \displaystyle \min_{\mathbf{y} \in \mathbb{R}^{y_M}} & \displaystyle \frac{1}{2}\|(D_{\hat{\boldsymbol{\pi}}} \otimes D_{\hat{\boldsymbol{\pi}}}^{-1})(I+K)\Pi \mathbf{y} - \mathbf{p}\|_2^2,  \\
        \text{s.t.} & (\hat{\boldsymbol{{\boldsymbol{\pi}}}}^{\top} \otimes I)(I+K)\Pi \mathbf{y} =\hat{\boldsymbol{\pi}}   \\
        &  \mathbf{y} \geq 0,
    \end{array}  
\end{equation}
where we used the fact that $\Pi^{\top} ((I+K)\Pi \mathbf{y}  ) = \operatorname{diag}(\mathbf{d}) \mathbf{y}$.
It is important to note that in this formulation, we have $n$ equality constraints and $y_M$ variables with $n \leq y_M$, which significantly reduces the dimensionality of the problem while preserving its essential structure. Finally, it is important to note that, thanks to the following Proposition \ref{eq:prop_max_rank}, the objective in \eqref{eq:p1_compact_final} is strongly convex.

\begin{proposition}\label{eq:prop_max_rank}
Given $\mathbb{R}^{n}  \ni \boldsymbol{\pi} > 0$, $K \in \mathbb{R}^{n^2 \times n^2}$ the commutation matrix for which $\operatorname{vec}(Y^\top) = K \operatorname{vec}(Y)$ for any $Y \in \mathbb{R}^{n \times n}$, and the projector $\Pi \in \mathbb{R}^{n^2 \times y_M}$ in~\eqref{eq:pi_projector} built from $M \in \{0,1\}^n$ irreducible, symmetric, and with $M_{i,i} = 1$, $i=1,\ldots,n$, then
\[
\operatorname{rank}\left((D_{\hat{\boldsymbol{\pi}}} \otimes D_{\hat{\boldsymbol{\pi}}}^{-1})(I+K)\Pi\right) = y_M.
\]
\end{proposition}

\begin{proof}
The matrix $(D_{\hat{\boldsymbol{\pi}}} \otimes D_{\hat{\boldsymbol{\pi}}}^{-1})$ is a diagonal matrix of size $n^2 \times n^2$ whose diagonal entries are $\nicefrac{\hat{\pi}_i}{\hat{\pi}_j} = \nicefrac{\sqrt{\pi}_i}{\sqrt{\pi}_j} > 0$, $\forall\,i,j=1,\ldots,n$, since $\boldsymbol{\pi} > 0$. Moreover, given $\mathbf{y}=\operatorname{vec}({Y})$, it holds $(I + K)\mathbf{y} = {0}$ iff $Y = - Y^\top$. Hence, to prove the thesis, we need to verify only that no vectorization of a skew-symmetric matrix is in the image of $\Pi$, since $\Pi$ has full rank $y_M$ by construction. On the other hand, this holds by construction since  $\Pi\mathbf{y} = \operatorname{vec}(S(\mathbf{y}))$ for $S(\cdot)$ the operator in~\eqref{eq:S_matrix_operator},  which returns a triangular matrix. The thesis follows by observing that the only skew-symmetric triangular matrix is the zero matrix. 
\end{proof}

\subsection{Formulation in QP Standard form}\label{sec:let_there_be_qp}

Quadratic programming (QP) provides a natural computational framework for the reversibilization problem formulated in the previous section. We seek a matrix $X \in \mathbb{R}^{n\times n}$, endowed with the prescribed sparsity pattern $\Omega_M$, that minimizes the Frobenius distance to the original transition matrix $P$ while satisfying the linear constraints of reversibility, stationarity, and stochasticity. Note that the objective function $\mathrm{J}(Y)$ is strongly convex (cf. Proposition~\ref{eq:prop_max_rank}) and therefore the program admits a unique solution.

For implementation with off-the-shelf solvers, it is sufficient to recognize that problem \eqref{eq:p1_compact_final} can be easily cast in the standard QP form
\begin{equation}\label{eq:standard_form_new}
\begin{aligned}
\min_{\mathbf{y} \in \mathbb{R}^{y_M}} \quad & \frac{1}{2} \mathbf{y}^\top Q \mathbf{y} + \mathbf{c}^\top \mathbf{y} \\
\text{subject to} \quad & A_{\text{eq}} \mathbf{y} = \mathbf{b}_{\text{eq}}, \\
&  \mathbf{y} \geq 0
\end{aligned}
\end{equation}
where $Q = ((D_{\hat{\boldsymbol{\pi}}} \otimes D_{\hat{\boldsymbol{\pi}}}^{-1})(I+K)\Pi)^\top  (D_{\hat{\boldsymbol{\pi}}} \otimes D_{\hat{\boldsymbol{\pi}}}^{-1})(I+K)\Pi \in \mathbb{R}^{{y_M} \times {y_M}}$ is a symmetric matrix defining the quadratic part of the objective,%
\[
\mathbf{c}^\top = -\mathbf{p}^\top (D_{\hat{\boldsymbol{\pi}}} \otimes D_{\hat{\boldsymbol{\pi}}}^{-1})(I+K)\Pi  \in \mathbb{R}^{y_M}
\]
defines the linear part, whereas $A_{\text{eq}}=(\hat{\boldsymbol{{\boldsymbol{\pi}}}}^{\top} \otimes I)(I+K)\Pi $
and $ \mathbf{b}_{eq}=\hat{\boldsymbol{\pi}} $ . %
Hence, Proposition~\ref{eq:prop_max_rank}, implies that $Q$ is positive definite and therefore the problem is strongly convex.

{We stress that in experiments presented in Section \ref{sec:benchmarkingQP}, to solve~\eqref{eq:standard_form_new}, we employ two widely  used solvers: MATLAB's \texttt{quadprog} and the \texttt{gurobi} QP solver. The former offers a simple interface and robust defaults suitable for prototyping and medium-scale instances; the latter provides a state-of-the-art engine with advanced presolve and parallel execution, which we found to deliver shorter runtimes on larger cases. Since the objective is strongly convex, the optimal solution is unique and therefore independent of the choice of algorithm, up to numerical precision.}
{Regarding the computational complexity, in both cases, the selected solver belongs to the class of Interior Point Methods (IPMs). This term refers to a broad family of algorithms that are known to enjoy polynomial worst-case iteration complexity with respect to the problem dimension; see, e.g., \cite{MR4865731}. Both MATLAB's \texttt{quadprog} and \texttt{gurobi} QP solvers provide mature and highly optimized implementations of this well-established algorithmic class. We also note that the dominant computational cost in each IPM iteration is typically the solution of the Newton system obtained by linearizing the perturbed Karush--Kuhn--Tucker optimality conditions.}

\subsection{The case of reducible Markov chains}\label{sec:the_reducible_case}

As pointed out in Definition~\ref{def:reversibility} and Remark~\ref{rmk:irreducible_and_reversible}, in Section~\ref{sec:problem_formulation} we provide a formulation limited to case of irreducible Markov chains. %
Nevertheless, more generally, a Markov chain may exhibit more than one ergodic class. We recall that an \emph{ergodic class} is a subset $C\subseteq\mathcal{V}$ of the state space with the three following properties: \emph{i)} closed, i.e. if $i\in C$ and $p_{ij}>0$, then $j\in C$; \emph{ii)} irreducible, that is for $i,j \in C$, there exists $\ell$ such that $[P^{\ell}]_{ij} >0$; \emph{iii)} aperiodic, which means that for $i \in C$, the $\text{gcd}\left\lbrace \ell \geq 1: [P^{\ell}]_{ii} > 0 \right\rbrace =1$.

A Markov chain with a single ergodic class has a unique stationary distribution $\boldsymbol{\pi}$, and if there are no transient states, then $\boldsymbol{\pi}$ is strictly positive, consistent with Definition~\ref{def:stationary_distribution}.
On the other hand, if the chain contains multiple ergodic classes and possibly some transient states, then the limiting behavior depends on the initial condition. In this case, the set of stationary distributions is not a singleton but rather a convex polytope spanned by the stationary distributions of the individual ergodic classes. More concretely, after a suitable permutation of states, the transition matrix $P$ can be expressed in block form as
\begin{equation}\label{eq:ergodic_decomposition}
    P = \begin{bmatrix}
P_1 & 0   & 0   & \cdots & 0      \\
0   & P_2 & 0   & \cdots & 0      \\
\vdots & \ddots & \ddots & \ddots & \vdots \\
0   & \cdots & 0 & P_E   & 0      \\
P_{T1} & P_{T2} & \cdots & P_{TE} & P_{TT}
\end{bmatrix},
\end{equation}
where each $P_i$ is a stochastic irreducible block associated with an ergodic class, and $T$ indexes the (possibly empty) set of transient states. In this case, every stationary distribution has the form
\begin{equation}\label{eq:stationary_reducible}
 \boldsymbol{\pi}^\top  = [\alpha_1 \boldsymbol{\pi}_1^\top,\ldots, \alpha_E \boldsymbol{\pi}_E^\top,\mathbf{0}_{|T|}^\top],    
\end{equation}
where $\alpha_i \in [0,1]$ with $\sum_{i=1}^E \alpha_i=1$, each $\boldsymbol{\pi}_i$ is the unique stationary distribution of $P_i$, and $\mathbf{0}_{|T|}$ is the zero vector corresponding to the transient states. Hence, uniqueness and positivity of the stationary distribution are guaranteed only when $E=1$ and $T=\emptyset$.

From the perspective of reversibility, the decomposition~\eqref{eq:ergodic_decomposition} highlights that detailed balance condition in Definition~\ref{def:reversibility} can hold within the support of a stationary distribution, that is $\operatorname{supp}(\boldsymbol{\pi}) = \{ i : \pi_i > 0 \}$, while for the indices  corresponding to the transient states we find
\[
\underbrace{\pi_i}_{=0} P_{ij} = \underbrace{\pi_j P_{ji}}_{=0}, 
\qquad \forall\, i \in T, \, j = 1,\ldots,n,
\]
and since both $\pi_i$ and $P_{ji}$ are zero for any $i \in T$, these equations are automatically satisfied. If $E=1$, reversibility is defined with respect to the unique strictly positive $\boldsymbol{\pi}$, as in Definition~\ref{def:reversibility}. If $E>1$, reversibility may still hold within each ergodic block $P_i$ with respect to its stationary distribution $\boldsymbol{\pi}_i$, but a global notion of reversibility for the entire chain requires specifying a convex combination as in~\eqref{eq:stationary_reducible}. In this sense, reversibility is inherently a property of ergodic classes and extends to the whole chain only relative to a chosen positive stationary distribution.

This analysis suggests that, in the general case, the appropriate approach to finding the closest reversible chain with a prescribed sparsity pattern is to remove the transient states and to restrict attention to the ergodic classes. The problem  can then be solved independently within each ergodic class, where the stationary distribution is strictly positive, and reversibility is well defined. An analogous treatment in the setting of dense transition matrices, where the nearest reversible chain is computed via Riemannian optimization, is given in~\cite[\S 3.2]{durastante2025riemannianoptimizationapproachfinding}.
To identify the transient states, we compute a stationary distribution $\boldsymbol{\pi}$ as $\boldsymbol{\pi}^\top = \lim_{k\to\infty} {\boldsymbol{\pi}^{(0)}}^\top P^k$, where $\boldsymbol{\pi}^{(0)}>0$ is a probability vector representing the probability distribution at time 0. If $P$ is reducible, the vector $\boldsymbol{\pi}$ depends on the initial distribution $\boldsymbol{\pi}^{(0)}$, but the indices corresponding to the zero entries of $\boldsymbol{\pi}$ uniquely identify the transient states.  We describe the algorithmic approach based on this idea in Section~\ref{sec:overall_algorithm}.

\subsection{The overall algorithm}\label{sec:overall_algorithm}

In light of Section~\ref{sec:the_reducible_case}, we first decompose the chain into ergodic classes and transient states via standard graph algorithms (e.g., Tarjan’s method~\cite{MR304178}; see also \cite{MR4874150}). We then restrict the optimization to the support of the stationary distribution, discarding transient states since they do not affect the detailed balance condition. This reduction yields independent subproblems on the ergodic classes, each of them well-posed (Proposition~\ref{prop:it_is_not_empty}) with a convex feasible set (Proposition~\ref{pro:convexity}). Problem \eqref{eq:original-rever} is a convex optimization problem with a strongly convex objective function and therefore admits a unique solution.
 
We summarize the entire strategy to obtain a sparse reversible Markov chain using the matrix nearness criterion in Algorithm~\ref{alg:summming_it_up}.
\begin{algorithm}[htbp]
\KwIn{$P \in \mathbb{R}^{n\times n}$ stochastic} %
\KwOut{$R \in \mathbb{R}^{n\times n}$ stochastic, reversible, sharing the stationary distribution of $P$, minimizing $\frac{1}{2} \|P-R\|_F^2$} %
\BlankLine
Compute $\boldsymbol{\pi}^\top = \lim_{k\to\infty} {\boldsymbol{\pi}^{(0)}}^\top P^k$, where $\boldsymbol{\pi}^{(0)}>0$, ${\boldsymbol{\pi}^{(0)}}^\top \mathbf{1} = \mathbf{1}$\;
Identify $\operatorname{supp}(\boldsymbol{\pi})$ and restrict $P$ and $\boldsymbol{\pi}$ to $\tilde{P}$ and $\tilde{\boldsymbol{\pi}}$, discarding transient states\;
Identify the number of ergodic classes $E$ of $\tilde{P}$ via Tarjan’s depth-first search method~\cite{MR304178}\;
\For{each ergodic class $\mathcal{C}$ from $1$ to $E$\label{line:parallel_for}}{
    \tcc{Denote by $\tilde{P}|_\mathcal{C}$ and $\tilde{\boldsymbol{\pi}}|_\mathcal{C}$ the restrictions to the current ergodic class}
    Build $D_{\hat{\boldsymbol{\pi}}} = \operatorname{diag}(\sqrt{\tilde{\boldsymbol{\pi}}|_\mathcal{C}})$\;
    Define $\hat{P} =  D_{\hat{\boldsymbol{\pi}}} \, \tilde{P}|_{\mathcal{C}} \, D_{\hat{\boldsymbol{\pi}}}^{-1}$\;
    Build the binary matrix $M \in \mathbb{S}(\hat{P}+\hat{P}^\top + I)$\;
    Solve the optimization problem~\eqref{eq:p1_compact_final} for $\mathbf{y}$\;
    Assemble $Y|_{\mathcal{C}}$ having the same pattern of $M$ and entries obtained from $\mathbf{y}$\;
    $R|_{\mathcal{C}} = D_{\hat{\boldsymbol{\pi}}}^{-1} X|_{\mathcal{C}} D_{\hat{\boldsymbol{\pi}}}$\;
}
Assemble $\tilde{R}$ by collecting $R|_{\mathcal{C}}$ for the $E$ ergodic classes\;
Reinsert the rows and columns corresponding to the zero entries of $\boldsymbol{\pi}$ into $\tilde{R}$ to obtain $R$ in the original ordering\;
The perturbation matrix $\Delta$ is then given by $\Delta = R - P$\;
\caption{Nearest reversible sparse Markov chain}\label{alg:summming_it_up}
\end{algorithm}

It is important to note that the reduction to ergodic classes makes computations independent across classes, so the loop at Line~\ref{line:parallel_for} of Algorithm~\ref{alg:summming_it_up} is embarrassingly parallel. To balance load across computing units, classes can be grouped so that their total number of unknowns is approximately uniform.

The procedure in Algorithm~\ref{alg:summming_it_up} is implemented by the MATLAB function \\ \texttt{nearest\_sparse\_reversible.m} in the repository \href{https://github.com/Cirdans-Home/sparse-reversible}{Cirdans-Home/sparse-reversible}. Given a possibly non-reversible transition matrix $P$, it parses options, computes (or accepts) the stationary distribution, the pattern, and delegates the optimization to one of the available QP solvers (\lstinline{'quadprog'}, \lstinline{'gurobi'}). When \lstinline{recurse_ergodic} is enabled, the state space is decomposed as in Section~\ref{sec:the_reducible_case} and each ergodic block is processed independently before reassembling the global reversible chain. The outputs comprise the reversible approximation \lstinline{R} and an auxiliary \lstinline{output} structure with solver diagnostics, perturbation statistics, and timings. %

\section{Numerical examples}\label{sec:numerical_examples}

This section presents a set of numerical experiments designed to assess the performance and effectiveness of the algorithm introduced in Section~\ref{sec:overall_algorithm}. The discussion is organized into two parts. The first part (Section~\ref{sec:benchmarkingQP}) focuses on synthetic, sparse Markov chains of increasing dimension, with the goal of validating the proposed method across a broad class of controlled test problems. The second part (Section~\ref{sec:transition_counts}) examines two real-world case studies drawn from molecular dynamics. In both cases, a discrete Markov chain is estimated from simulation data to describe transitions between the various spatial conformations of a complex molecule. Although the underlying physical process is time-reversible, the estimated transition matrices are generally not reversible due to statistical and sampling errors, and therefore must be post-processed to obtain a reversible approximation.

All numerical experiments presented in this section were performed on a laptop equipped with an \mbox{Intel\textsuperscript{\textregistered} Core\texttrademark{} i9-14900HX} processor (24 cores, 32 hardware threads) and $64$~GB of RAM. The operating system is \mbox{Ubuntu 24.04.2 LTS}, and all computations were carried out using \mbox{MATLAB~R2025b}. The MATLAB interface of version \mbox{Gurobi~12.0.1} has been used.

All numerical experiments presented in this work are fully reproducible using the code available at the GitHub repository \href{https://github.com/Cirdans-Home/sparse-reversible}{Cirdans-Home/sparse-reversible}.

\subsection{Bechmarking the QP approaches}\label{sec:benchmarkingQP}

In this section, we compare different solution methods for the QP problem arising from the computation of the reversible matrix closest to a given stochastic one. Specifically, we compare MATLAB's  \lstinline{quadprog} solver~\cite{Altman01011999}, and the commercial \texttt{gurobi} QP solver\footnote{See the website \href{https://www.gurobi.com/}{www.gurobi.com} to get the solver along with an academic license.}.

\subsubsection{Generation of Sparse Random Markov Chains}
\label{subsec:testcase_generation}

To evaluate the proposed methods, we construct a collection of sparse, non-reversible Markov chains with randomly varying state-space sizes. All experiments are reproducible by fixing the random seed.

Specifically, we generate $N = 100$ independent test cases. For each test case, the number of states $n$ is drawn uniformly at random from the interval $[n_{\min}, n_{\max}]$, with $n_{\min} = 100$ and $n_{\max} = 300$. A sparse random matrix $P \in \mathbb{R}^{n \times n}$ is then constructed by sampling $\mathrm{nnz}(P) = \lfloor \alpha n \rfloor$ nonzero entries at uniformly random positions, where $\alpha = 5$ controls the sparsity level. The resulting matrix is therefore sparse, with an average of $\mathcal{O}(n)$ nonzero entries. To ensure that the associated Markov chain is irreducible, we interpret the nonzero pattern of $P$ as the adjacency matrix of a directed graph. We compute the strongly connected components of this graph and restrict $P$ to the largest strongly connected component, thereby guaranteeing irreducibility of the resulting transition structure. Finally, the matrix is row-normalized to obtain a valid transition matrix,
\[
\tilde{P}_{ij} = \frac{P_{ij}}{\sum_{k} P_{ik}}, \quad i = 1, \dots, n,
\]
ensuring that each row sums to one. The resulting matrix $\tilde{P}$ defines a sparse, irreducible, and generally non-reversible Markov chain. This procedure is repeated independently for each test case.

We apply the proposed optimization framework to compute the nearest sparse reversible Markov chain using two different optimization algorithms. We begin by reporting, in Fig.~\ref{fig:nnz_manufactured1}, a comparison between 
\begin{figure}[htbp]
    \sidecaption
    \includegraphics[width=7.1cm]{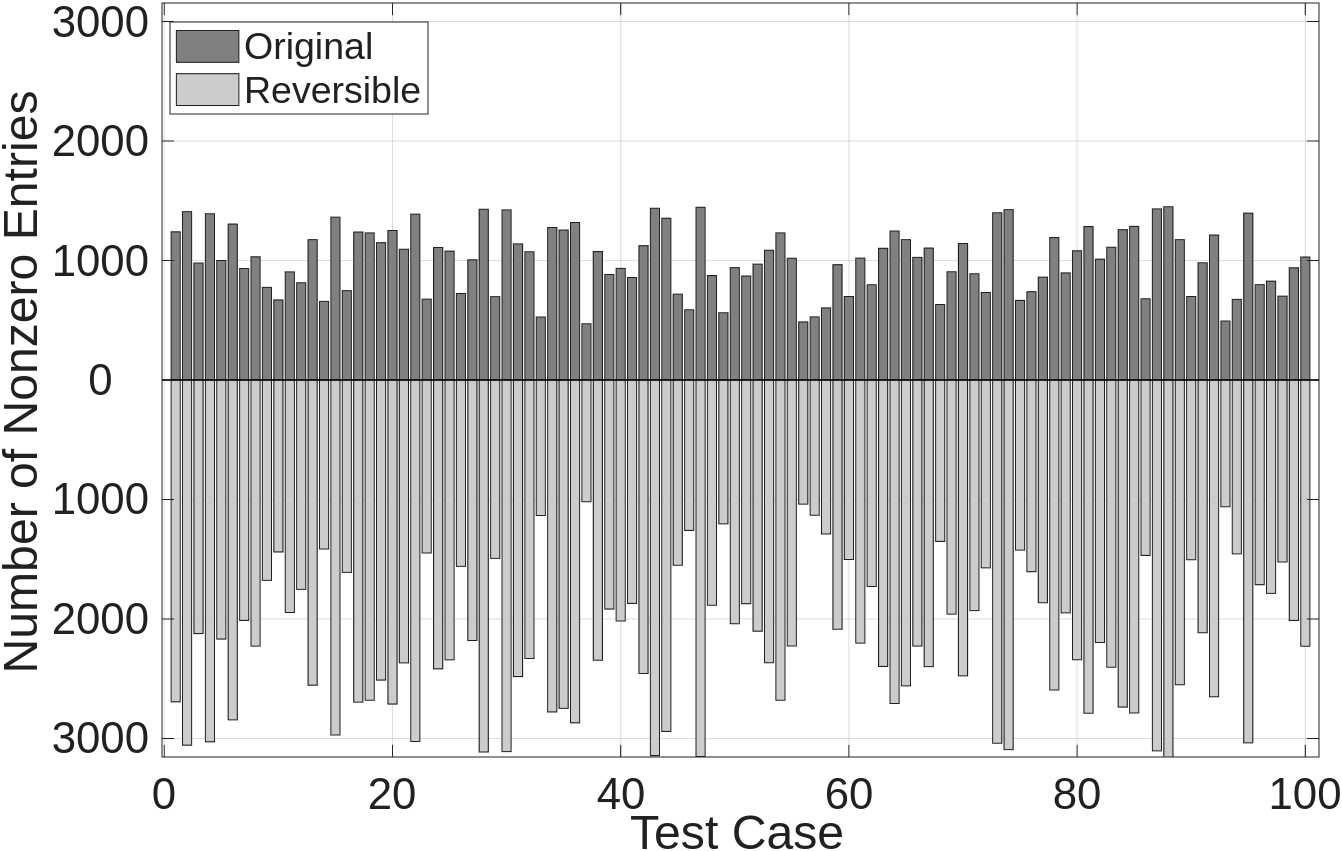}
    \caption{Diverging bar plots showing the number of nonzero entries in the original Markov chain (positive values) and in the recovered reversible matrix (negative values). Since the underlying optimization problem is \emph{strongly convex}, the resulting sparsity pattern is independent of the optimization algorithm employed.}
    \label{fig:nnz_manufactured1}
\end{figure}
the number of nonzero entries of the original sparse transition matrix of the Markov chain and those of the corresponding reversible matrix. The reversible matrices consistently exhibit a substantially larger number of nonzero entries than their original counterparts. This behavior is expected, as enforcing reversibility introduces additional coupling constraints between forward and backward transitions, which generally leads to a denser sparsity pattern; see again Proposition~\ref{pro:convexity}. While the size of the original matrices varies significantly across test cases—reflecting the randomized construction—the relative increase in the number of nonzero entries remains systematic and stable. This indicates that the sparsity–reversibility trade-off is largely independent of the specific realization and is instead an intrinsic consequence of the reversibility constraint.

In Fig.~\ref{fig:algorithm-performances} we report the algorithmic performances of the procedure with respect to the different optimization methods used inside.
\begin{figure}[htbp]
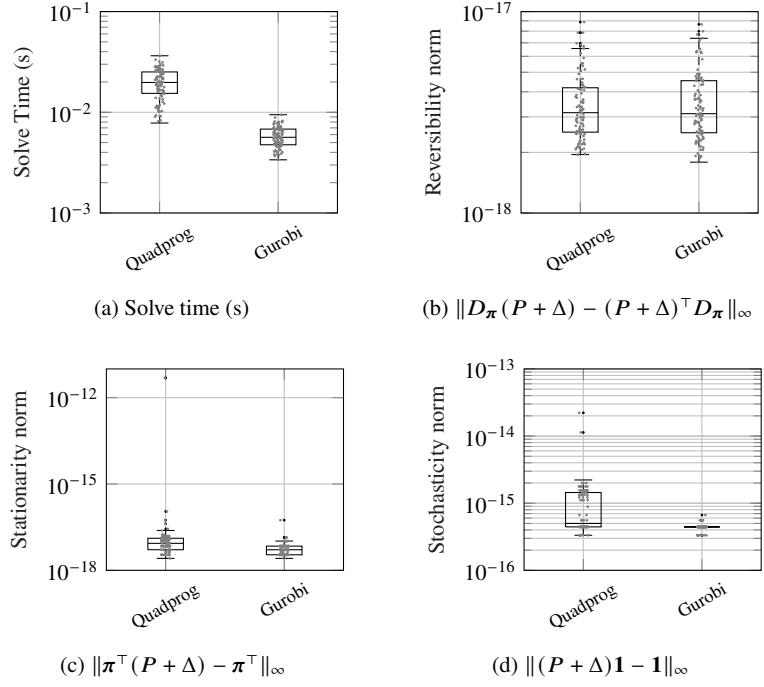

    \centering
    \subfloat[Solve time (\SI{}{\second})]{\definecolor{mycolor1}{rgb}{0.12941,0.12941,0.12941}%
\begin{tikzpicture}

\begin{axis}[%
width=0.26\columnwidth,
height=0.221\columnwidth,
at={(0\columnwidth,0\columnwidth)},
scale only axis,
unbounded coords=jump,
xmin=0.5,
xmax=2.5,
xtick={1,2},
xticklabels={{Quadprog},{Gurobi}},
xticklabel style={rotate=30,font=\scriptsize},
ymin=0.001,
ymax=0.1,
ylabel style={font=\color{mycolor1}},
ylabel={Solve Time (\SI{}{\second})},
axis background/.style={fill=white},
title style={font=\bfseries\color{mycolor1}},
xmajorgrids,
ymajorgrids,
ymode=log,
]
\addplot [color=black, dashed, forget plot]
  table[row sep=crcr]{%
1	0.025157\\
1	0.036412\\
};
\addplot [color=black, dashed, forget plot]
  table[row sep=crcr]{%
2	0.0068065\\
2	0.00946\\
};
\addplot [color=black, dashed, forget plot]
  table[row sep=crcr]{%
1	0.007821\\
1	0.0154065\\
};
\addplot [color=black, dashed, forget plot]
  table[row sep=crcr]{%
2	0.003373\\
2	0.0047575\\
};
\addplot [color=black, forget plot]
  table[row sep=crcr]{%
0.925	0.036412\\
1.075	0.036412\\
};
\addplot [color=black, forget plot]
  table[row sep=crcr]{%
1.925	0.00946\\
2.075	0.00946\\
};
\addplot [color=black, forget plot]
  table[row sep=crcr]{%
0.925	0.007821\\
1.075	0.007821\\
};
\addplot [color=black, forget plot]
  table[row sep=crcr]{%
1.925	0.003373\\
2.075	0.003373\\
};
\addplot [color=black, forget plot]
  table[row sep=crcr]{%
0.85	0.0154065\\
0.85	0.025157\\
1.15	0.025157\\
1.15	0.0154065\\
0.85	0.0154065\\
};
\addplot [color=black, forget plot]
  table[row sep=crcr]{%
1.85	0.0047575\\
1.85	0.0068065\\
2.15	0.0068065\\
2.15	0.0047575\\
1.85	0.0047575\\
};
\addplot [color=black, forget plot]
  table[row sep=crcr]{%
0.85	0.019778\\
1.15	0.019778\\
};
\addplot [color=black, forget plot]
  table[row sep=crcr]{%
1.85	0.0056425\\
2.15	0.0056425\\
};
\addplot [color=mycolor1, only marks, mark=*,  mark options={solid, fill=black, draw=black}, forget plot]
  table[row sep=crcr]{%
nan	nan\\
};
\addplot [color=mycolor1, only marks, mark=*, mark options={solid, fill=black, draw=black}, forget plot]
  table[row sep=crcr]{%
nan	nan\\
};
\addplot [color=gray, only marks, mark size=0.3pt, mark=*, mark options={solid, gray}, forget plot]
  table[row sep=crcr]{%
1.00385978471557	0.036412\\
0.980050459117321	0.028024\\
1.01780889990135	0.018907\\
0.969144758114728	0.028688\\
1.00919001412091	0.019159\\
0.980973380229536	0.022541\\
0.990751318438673	0.018282\\
1.01052088197084	0.0195\\
1.02386111132032	0.025367\\
1.00015972100556	0.010163\\
1.03774972232651	0.015269\\
1.01791536580271	0.019714\\
1.02518136227669	0.026661\\
0.993221808739318	0.011513\\
0.98400979532426	0.027438\\
0.984740692245418	0.016035\\
1.00707915422506	0.025191\\
1.00662869953518	0.029645\\
1.00930387325921	0.022655\\
1.01069350963576	0.023735\\
0.977808594130149	0.018904\\
0.999268962556285	0.028618\\
0.987044356274561	0.010971\\
0.979433145418065	0.021201\\
0.976063413545881	0.022491\\
0.966266253694456	0.0139\\
0.974488362577487	0.018756\\
1.00007704021611	0.025701\\
0.975983398723119	0.011566\\
1.0042337315709	0.030001\\
0.996370671829001	0.022587\\
1.02299245805992	0.022115\\
0.960838691308591	0.009023\\
0.998929077400267	0.025123\\
1.02493229765308	0.025617\\
1.00647579361855	0.027075\\
1.02347713448146	0.007821\\
1.00712522954782	0.020083\\
0.986866425684556	0.016633\\
0.987595857385667	0.018394\\
1.03511468436457	0.016264\\
0.970697829014722	0.021557\\
0.999616439350053	0.026066\\
0.966540702785524	0.033286\\
0.998448655358684	0.013699\\
1.00913789838361	0.009144\\
1.02497041430518	0.031222\\
0.982743835286662	0.017001\\
0.970005745600626	0.009326\\
0.977089667346393	0.018337\\
1.02521983714052	0.016672\\
0.983933549467324	0.017389\\
1.03163112760905	0.021575\\
0.961358740644571	0.021023\\
1.02847860821868	0.020015\\
0.990395643349385	0.008151\\
1.00756090860857	0.008335\\
1.000128633602	0.009138\\
0.986439967437012	0.018429\\
0.966385097638777	0.011129\\
0.97283251449158	0.021902\\
1.01431140492948	0.020101\\
0.970825844948209	0.019506\\
1.01435419215709	0.024123\\
1.03200986863792	0.024052\\
0.983631973682737	0.019842\\
0.967006191572901	0.026197\\
0.992284159831451	0.010112\\
1.03850008073703	0.017793\\
1.01841199075767	0.022204\\
1.00505151936099	0.016616\\
1.02569252215779	0.013463\\
0.960800006172267	0.028273\\
0.993886239045432	0.028277\\
0.986419238281224	0.010592\\
1.01753474425746	0.014409\\
0.992835110032679	0.016963\\
1.02130096306017	0.023899\\
1.03793379887918	0.015544\\
1.02352987027533	0.021642\\
0.983444633277002	0.026301\\
1.01339726605252	0.018969\\
1.02270417096289	0.021507\\
0.969356855847851	0.025369\\
0.982885244641703	0.025895\\
0.992182794342685	0.011048\\
0.997657345503227	0.031897\\
1.03002039119339	0.029204\\
1.02811021590121	0.025092\\
1.02513935087024	0.012745\\
0.973089427981332	0.018642\\
1.00856287884982	0.025008\\
1.01987306513529	0.010263\\
1.03899056367481	0.012222\\
0.988165272075718	0.028598\\
1.01470063422226	0.014848\\
1.02468368939145	0.015678\\
1.00028931972054	0.01183\\
1.01274463545462	0.018229\\
0.995801588627275	0.020238\\
};
\addplot [color=gray, only marks, mark size=0.3pt, mark=*, mark options={solid, gray}, forget plot]
  table[row sep=crcr]{%
1.96539802604903	0.007127\\
1.97370467643693	0.008863\\
1.99597764183825	0.005337\\
1.99385003710566	0.007793\\
2.03284981694496	0.006266\\
2.02428867606193	0.007146\\
1.9628148706448	0.005535\\
1.9618614003536	0.005952\\
2.03649823422501	0.005053\\
2.01313755786716	0.004237\\
1.98939558705531	0.00507\\
2.0089629814998	0.005044\\
2.03262879681653	0.006365\\
1.97644369440714	0.004121\\
1.96526750998559	0.007693\\
1.98701930538149	0.004543\\
1.96087761227316	0.00696\\
2.01450046973473	0.007025\\
1.97216109870357	0.006203\\
1.97228795385923	0.006612\\
2.03976636356282	0.006551\\
2.00240395946478	0.007576\\
2.00564733203853	0.004182\\
1.99233764152116	0.005818\\
1.96184602077116	0.00625\\
1.98116134832305	0.004683\\
2.03559686700349	0.005284\\
1.96592252686028	0.007755\\
2.01414752421353	0.004715\\
2.00524537718393	0.008022\\
2.03334877440323	0.005988\\
2.00392540654515	0.006152\\
1.99141132431174	0.003821\\
1.99000310417645	0.006815\\
2.00750250852987	0.006656\\
1.99519686308025	0.007714\\
2.00240036993088	0.003373\\
2.01562624069761	0.005668\\
1.98407393785962	0.004695\\
2.02056875430789	0.005635\\
1.96273693957647	0.005294\\
2.00102179711322	0.006192\\
2.03686655507329	0.0082\\
2.0047592202944	0.007574\\
1.97811283416857	0.004173\\
2.03436166967602	0.004019\\
1.97036040672915	0.008112\\
1.97809762572805	0.005085\\
2.02774623817437	0.003896\\
1.97418828614072	0.005056\\
1.98888776211894	0.004994\\
1.99340602476012	0.005365\\
1.98082757805423	0.006031\\
1.96765051278089	0.006833\\
2.03348836959213	0.006125\\
1.9669362713895	0.003802\\
1.97226988728838	0.003675\\
2.00910080794064	0.003778\\
1.98672842237643	0.005025\\
2.01949447700174	0.004232\\
1.97816943575594	0.005573\\
1.96672335635005	0.004859\\
1.9983590234648	0.006002\\
2.03594513319349	0.00666\\
1.96333056737442	0.006416\\
2.02036608978739	0.00547\\
2.01079394536687	0.005782\\
2.02789241608903	0.004651\\
2.02890676511616	0.005378\\
1.98088539627878	0.006124\\
2.01413482641656	0.004929\\
2.00924255298055	0.004306\\
2.03362386166627	0.007972\\
2.00681240220928	0.007921\\
1.98488587002925	0.004093\\
2.02072207969875	0.004591\\
2.02491236233107	0.005265\\
1.97077449394735	0.006461\\
2.00386263459519	0.004974\\
2.03179648418807	0.006163\\
2.02112095689757	0.007272\\
1.98477222261616	0.005048\\
1.99273540561102	0.006798\\
1.96028792559672	0.006873\\
1.99416521465189	0.007032\\
1.96599767091441	0.004086\\
2.03063063230378	0.00946\\
2.03261390528361	0.008125\\
2.01085574058542	0.006945\\
1.98152511915596	0.0048\\
1.98752997043141	0.00565\\
2.02524992519006	0.006672\\
1.99628650198579	0.003917\\
2.00931048318498	0.004557\\
1.96166296485906	0.007597\\
1.99477553968371	0.004611\\
2.02954179423482	0.004814\\
1.98932431034777	0.004624\\
1.99201501546364	0.005404\\
1.97533399362647	0.005609\\
};
\end{axis}
\end{tikzpicture}%
}\hfil
    \subfloat[$\|D_{\boldsymbol{\pi}}(P+\Delta) - (P+\Delta)^\top D_{\boldsymbol{\pi}}\|_\infty$]{\input{reversibility}}
    
    \subfloat[$\|\boldsymbol{\pi}^\top (P+\Delta) - \boldsymbol{\pi}^\top\|_\infty$]{\input{stationarity}}\hfil
    \subfloat[$\|(P+\Delta)\mathbf{1} - \mathbf{1}\|_\infty$]{\input{stochastic}}
        
    \caption{Boxplots summarizing the distribution of solution times and condition numbers across all test cases, grouped by solver. Each boxplot displays the median and interquartile range, with individual test cases shown as jittered dots.} %
    \label{fig:algorithm-performances}
\end{figure}
These are boxplots summarizing the empirical distribution of the reported quantities---solution times, distance to reversibility, being $\boldsymbol{\pi}$ the stationary distribution, and being stochastic---across all test cases for each solver. For a given method, the central line of the box denotes the median value, while the lower and upper edges of the box correspond to the first and third quartiles, respectively. The whiskers extend to the most extreme data points that are not classified as outliers according to the standard $1.5$ interquartile range criterion. Individual data points are overlaid as jittered dots, where each dot represents the outcome of a single test case. The horizontal jitter is introduced solely to avoid overlap and does not convey additional information.

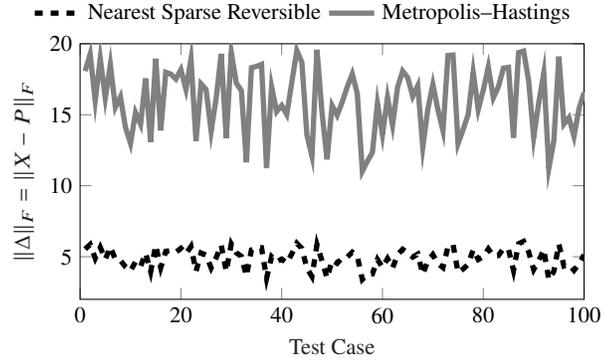
\begin{figure}[htbp]
    \sidecaption
    \definecolor{mycolor1}{rgb}{0.06600,0.44300,0.74500}%
\definecolor{mycolor2}{rgb}{0.86600,0.32900,0.00000}%
\definecolor{mycolor3}{rgb}{0.12941,0.12941,0.12941}%
\begin{tikzpicture}
\begin{axis}[%
width=6.7cm,
height=3.4cm,
at={(1.293in,0.678in)},
scale only axis,
xmin=0,
xmax=100,
xlabel style={font=\color{mycolor3}},
xlabel={Test Case},
ymin=2,
ymax=20,
ylabel style={font=\color{mycolor3}},
ylabel={$\|\Delta\|_F = \| X - P \|_F$},
axis background/.style={fill=white},
legend columns=2,
legend style={legend cell align=left, align=left, at={(1,1.2)}, fill=none, draw=none}
]
\addplot [color=black, line width=2.0pt, dashed]
  table[row sep=crcr]{%
1	5.53857037388604\\
2	5.86371702257808\\
3	4.89125840371233\\
4	5.65843502399002\\
5	4.81863751377362\\
6	5.68623028774833\\
7	4.81180109313419\\
8	4.74043818874274\\
9	4.29323401342511\\
10	3.99898548172286\\
11	4.6517738255519\\
12	4.32104402184159\\
13	5.25015272099351\\
14	3.94635999365848\\
15	5.79975334927256\\
16	4.22810057775898\\
17	5.38074141242321\\
18	5.43249056656129\\
19	5.30938496063645\\
20	5.59410344533436\\
21	5.19726828273975\\
22	5.71676532864453\\
23	3.97282523672787\\
24	5.26374554039704\\
25	5.11503054690988\\
26	4.20933999586243\\
27	4.93345755935087\\
28	5.8176162389542\\
29	3.966043305115\\
30	5.9009258184725\\
31	5.27832187145672\\
32	5.02747827215063\\
33	3.63688179941014\\
34	5.48323105513329\\
35	5.63862756316816\\
36	5.63400153031848\\
37	3.48213432543887\\
38	4.99642960135533\\
39	4.64045548478617\\
40	4.83310343480023\\
41	4.55885716544491\\
42	5.19736026770626\\
43	5.94787276531657\\
44	5.58530228322954\\
45	4.19401716604369\\
46	3.60753732999577\\
47	5.85053049140317\\
48	4.70388476048018\\
49	3.49349081986847\\
50	4.75732099144097\\
51	4.53380581021066\\
52	4.91831543618193\\
53	4.97438597165385\\
54	5.31915318970044\\
55	5.04382749959478\\
56	3.39966849216811\\
57	3.73400872888359\\
58	3.79077154534487\\
59	4.80477831256683\\
60	4.2427202865809\\
61	4.86874555654823\\
62	4.208783611175\\
63	5.14260125131299\\
64	5.52443149516112\\
65	5.27790384995227\\
66	4.97427078122368\\
67	5.18499168291726\\
68	3.76723799467569\\
69	4.70545906442706\\
70	5.14663113079822\\
71	4.52681700694493\\
72	4.23586301672142\\
73	5.70966914958969\\
74	5.86433780572314\\
75	3.8928014428801\\
76	4.13603041357954\\
77	4.45723542592839\\
78	5.50419596914491\\
79	4.80382498582509\\
80	5.00115145290249\\
81	5.54774294471685\\
82	5.04829194202732\\
83	5.08021649655318\\
84	5.58429158137275\\
85	5.52297345729335\\
86	4.03470872019939\\
87	5.91548760305469\\
88	6.07963289526433\\
89	5.33807756018774\\
90	4.09598446034846\\
91	4.70982852919781\\
92	5.50673278471539\\
93	3.38818723051338\\
94	3.97053470132786\\
95	5.88209523050184\\
96	4.34553491594213\\
97	4.45863874962699\\
98	4.03808219666932\\
99	4.49280916038439\\
100	5.13136159617432\\
};
\addlegendentry{Nearest Sparse Reversible}

\addplot [color=gray, line width=2.0pt]
  table[row sep=crcr]{%
1	18.0777348590835\\
2	19.2732780978993\\
3	15.9896192718463\\
4	18.9654497461247\\
5	16.1818534098997\\
6	18.52688053851\\
7	15.5205085877175\\
8	16.2605042956103\\
9	14.1107260880371\\
10	12.9843830376508\\
11	15.1527703959497\\
12	14.3284937094218\\
13	17.5611865673969\\
14	13.0880040247695\\
15	18.9465864696636\\
16	13.8950071366452\\
17	17.9902406161504\\
18	17.8588052475569\\
19	17.5156278502894\\
20	18.1782785768982\\
21	16.8019583800096\\
22	18.9202156172474\\
23	13.1596315815981\\
24	17.1780080647731\\
25	16.7786397134291\\
26	13.840591439629\\
27	16.1783278297151\\
28	19.2853392130198\\
29	13.3482179956547\\
30	19.3815084715505\\
31	17.2207682802025\\
32	16.7058214062463\\
33	11.6481374219864\\
34	18.3134320741365\\
35	18.4053962591986\\
36	18.5443958854245\\
37	11.2388216808996\\
38	16.9647007900434\\
39	15.2122347037813\\
40	15.6524490505585\\
41	15.0184287857413\\
42	17.2578102721369\\
43	19.5257498243469\\
44	18.6941146076061\\
45	13.5762834980468\\
46	12.0979427589769\\
47	19.5798521319121\\
48	15.1663536272402\\
49	11.8527870247846\\
50	15.6467360006813\\
51	14.9700417894958\\
52	15.8910834947241\\
53	16.9569790669604\\
54	17.8718051203675\\
55	16.4942217863631\\
56	11.0219779128462\\
57	11.7310032735341\\
58	12.3440187407171\\
59	15.8032309346972\\
60	13.6113226743422\\
61	16.2169188040585\\
62	14.2827450976225\\
63	16.94877787708\\
64	18.0762349024617\\
65	17.6232030364953\\
66	16.2968257724486\\
67	17.0492244195819\\
68	12.7238966849824\\
69	15.3961881015113\\
70	17.1312060027107\\
71	15.3541155732396\\
72	13.8853324120754\\
73	19.1736134192928\\
74	19.2085197937816\\
75	12.9703676507551\\
76	13.9538279726161\\
77	15.0110479529891\\
78	17.9926978697219\\
79	15.3690201483062\\
80	16.7434786899899\\
81	18.3629117320655\\
82	16.1938568258002\\
83	16.9731491070834\\
84	18.2838143337035\\
85	18.3354926211392\\
86	13.3803578812971\\
87	19.3723384568812\\
88	19.4847494928245\\
89	17.3961710586456\\
90	13.4176621746226\\
91	15.8350788913626\\
92	17.9324558370889\\
93	11.2252555138282\\
94	13.1496071412633\\
95	19.0976796679592\\
96	14.3433299216176\\
97	14.7934718425427\\
98	13.4881369912032\\
99	15.4357291252159\\
100	16.561714083897\\
};
\addlegendentry{Metropolis--Hastings}

\end{axis}
\end{tikzpicture}%
    \caption{Comparison of the norm of the perturbation $\Delta = X - P$ between the $X$ obtained from $P$ by applying Metropolis--Hastings reversibilization in~\eqref{eq:reversibilization}-\eqref{eq:reversibilization_metropolis}, and the nearest-sparse reversible obtained via our optimization procedure as in Algorithm~\ref{alg:summming_it_up} while employing the \texttt{gurobi} QP solver.\label{fig:norm_distance_to_metropolis}}
\end{figure}
In Figure~\ref{fig:norm_distance_to_metropolis}, we compare the Frobenius norm of the perturbation $\Delta$ (i.e., the distance to reversibility, with sparsity constraints) obtained through our optimization procedure with the distance at which the Metropolis–Hastings reversibilization (Equations~\eqref{eq:reversibilization}–\eqref{eq:reversibilization_metropolis}) of the original matrix sits. We observe that the Metropolis–Hastings approach produces a reversible matrix that is further apart. In terms of algorithmic performance, the two optimization methods exhibit notable differences across the evaluated metrics. In terms of computational efficiency, \texttt{gurobi} emerges as the clear winner, achieving solve times approximately three to four times faster than \texttt{guadprog}, with median times around $\SI{0.007}{\second}$ compared to \texttt{guadprog}'s $\SI{0.023}{\second}$. Concerning constraint satisfaction, the two solvers maintain excellent accuracy in enforcing the reversibility condition, with stationarity metrics showing negligible violations across all methods. However, the stochasticity constraint reveals a more pronounced difference: \texttt{gurobi} enforce this constraint with minimal violations at machine precision levels ($\sim 4.44 \times 10^{-16}$), while MATLAB's \texttt{guadprog} method exhibits larger violations. Overall, \texttt{gurobi} provides the best results.

To assess the scalability of the proposed approach, we examine the computational time as a function of the problem size. Fig.~\ref{fig:time-scaling} illustrates the solve time scaling behavior across test cases of varying dimensions. The horizontal axis represents the size of the Markov chain (i.e., the number of states), while the top axis indicates the number of nonzero entries in the original sparse transition matrix, reflecting the sparsity structure of each instance.
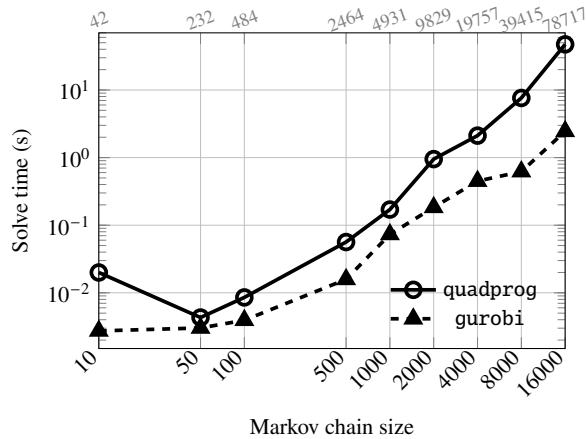
\begin{figure}[htbp]
    \sidecaption
    \begin{tikzpicture}

\begin{axis}[%
width=6.2cm,
height=4.2cm,
scale only axis,
xmode=log,
ymode=log,
xmin=10,
xmax=16000,
xtick={10,50,100,500,1000,2000,4000,8000,16000},
xticklabels={10,50,100,500,1000,2000,4000,8000,16000},
x tick label style={
    font=\small,
    rotate=45,
    anchor=east
},
xminorticks=true,
ymin=0.0015,
ymax=70,
yminorticks=true,
xlabel={Markov chain size},
ylabel={Solve time (s)},
axis background/.style={fill=white},
major grid style={line width=.3pt, draw=gray!50},
minor grid style={line width=.1pt, draw=gray!30},
xmajorgrids,
ymajorgrids,
legend style={
    at={(0.97,0.03)},
    anchor=south east,
    draw=none,
    fill=none,
    font=\small
},
tick label style={font=\small},
label style={font=\small},
clip=false
]

\addplot[
    black,
    solid,
    line width=1.4pt,
    mark=o,
    mark options={solid, fill=black},
    mark size=2.8pt
]
table{
10	0.019989
50	0.00431
100	0.00855
500	0.056364
1000	0.170433
2000	0.9496
4000	2.114131
8000	7.620537
16000	47.292649
};
\addlegendentry{\texttt{quadprog}}

\addplot[
    black,
    dashed,
    line width=1.4pt,
    mark=triangle*,
    mark options={solid, fill=black},
    mark size=3pt
]
table{
10	0.002732
50	0.003021
100	0.003931
500	0.015819
1000	0.07313
2000	0.183832
4000	0.448846
8000	0.620928
16000	2.448867
};
\addlegendentry{\texttt{gurobi}}

\node[font=\scriptsize, text=gray, rotate=20]
    at (axis cs:10,110) {42};
\node[font=\scriptsize, text=gray, rotate=20]
    at (axis cs:50,110) {232};
\node[font=\scriptsize, text=gray, rotate=20]
    at (axis cs:100,110) {484};
\node[font=\scriptsize, text=gray, rotate=20]
    at (axis cs:500,110) {2464};
\node[font=\scriptsize, text=gray, rotate=20]
    at (axis cs:1000,110) {4931};
\node[font=\scriptsize, text=gray, rotate=20]
    at (axis cs:2000,110) {9829};
\node[font=\scriptsize, text=gray, rotate=20]
    at (axis cs:4000,110) {19757};
\node[font=\scriptsize, text=gray, rotate=20]
    at (axis cs:8000,110) {39415};
\node[font=\scriptsize, text=gray, rotate=20]
    at (axis cs:16000,110) {78717};

\end{axis}
\end{tikzpicture}
    \caption{Solve time as a function of the Markov chain size (bottom axis) and the number of nonzero entries in the original sparse transition matrix (top axis). The two optimization algorithms demonstrate similar scaling behavior, with computational cost increasing with both the problem dimension and the sparsity level.}
    \label{fig:time-scaling}
\end{figure}
The results reveal that both algorithms exhibit subquadratic scaling with respect to the problem size, which is favorable for large-scale applications. \texttt{gurobi} consistently outperforms MATLAB's \texttt{quadrprog} across all problem sizes, maintaining its computational advantage even as the dimension increases. The correlation between solve time and the number of nonzero entries underscores the importance of sparsity in the original transition matrix: denser matrices lead to more complex optimization problems with a larger number of active constraints, consequently requiring longer solution times. For small to moderately sized problems (up to approximately $100$ states with a few hundred nonzero entries), both methods deliver solutions in well under one second, demonstrating the viability of the approach for real-time or near-real-time applications in moderate-dimensional settings.

\subsubsection{Comparison with the Riemannian approach for dense Markov Chains}

To complete the benchmark, we compare Alg.~\ref{alg:summming_it_up}---employing \texttt{gurobi}---with the Riemannian algorithm proposed in~\cite{durastante2025riemannianoptimizationapproachfinding} for computing the nearest reversible matrix without structural constraints. In Fig.~\ref{fig:comparison_riemannian}, the left panel displays the sparsity pattern of the transition matrix of the original Markov chain, which has size $251 \times 251$ with $1240$ nonzero entries. The matrix is generated according to the same procedure described in Section~\ref{subsec:testcase_generation}.
\begin{figure}[htb]
    \centering
    \input{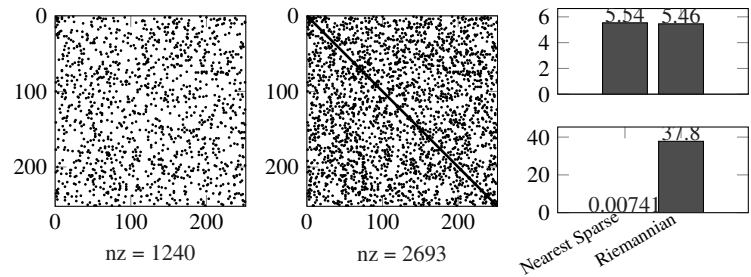}
    \caption{Comparison with the Riemannian solver. \emph{Left:} sparsity pattern of the transition matrix of the initial Markov chain. \emph{Center:} sparsity pattern of the nearest sparse reversible transition matrix obtained with Alg.~\ref{alg:summming_it_up}. \emph{Right:} the Riemannian solver produces a dense matrix (not shown as a sparsity pattern). The top bar plot reports the perturbation norm, while the bottom bar plot shows the computational time required, respectively, by the approach here proposed (Sparse Nearest) and by the Riemannian solver.}
    \label{fig:comparison_riemannian}
\end{figure}
The central panel shows the sparsity pattern of the \emph{nearest sparse reversible chain} computed with Alg.~\ref{alg:summming_it_up}. The number of nonzero entries increases to $2693$, which is comparable to the result reported in Fig.~\ref{fig:nnz_manufactured1}. By contrast, the Riemannian solver produces a dense matrix by construction.

We next compare the norm of the perturbations, reported in the top bar plot. The obtained values are approximately $5.54$ for the \emph{sparse} solver and $5.46$ for the Riemannian solver. This ordering is expected, since the Riemannian approach allows modifications on all entries, thus enjoying greater flexibility. Nevertheless, restricting the admissible nonzero pattern does not significantly deteriorate the achieved norm.

Finally, we observe that the reduced number of nonzero entries makes the sparse solver substantially faster: it computes the solution in $7.41 \times 10^{-2}$ seconds, compared to $37.8$ seconds required by the Riemannian solver.

\subsection{Direct estimation of Markov chains from transition counts}\label{sec:transition_counts}

We consider a system governed by the Langevin stochastic differential equation~\cite{PhysRevE.93.033307}
\begin{equation}\label{eq:sde}
    \dot{x} = -\frac{\partial U(x)}{\partial x} + \xi(t), \qquad x \in \mathbb{R},
\end{equation}
where $U(x)$ is a potential and $\xi(t)$ denotes Gaussian white noise with covariance 
$\langle \xi(t)\,\xi(s) \rangle = \sigma^2 \delta(t-s)$.  
For numerical simulations, the Euler--Maruyama scheme~\cite{Kloeden1992,MR1872387} is employed, yielding
\begin{equation}\label{eq:Euler--Maruyama}
    x(t+\Delta t) = x(t) - \frac{\partial U(x(t))}{\partial x}\,\Delta t 
    + \sigma \sqrt{\Delta t}\,\eta(t),
\end{equation}
with $\eta(t)$ independent standard Gaussian variables.  

The resulting trajectory $\{x(t_n)\}_{n=0}^N$ is coarse-grained by partitioning the state space into $B$ disjoint sets 
$\{\Theta_i\}_{i=1}^B$, which define the states of a Markov model.  
From the trajectory one forms the count matrix $C$, where $C_{ij}$ denotes the number of transitions from $\Theta_i$ to $\Theta_j$.  
The transition probability matrix $P$ is then estimated via row normalization,
\[
P_{ij} = \frac{C_{ij}}{\sum_{k=1}^B C_{ik}}, \qquad P \mathbf{1} = \mathbf{1}.
\]
Although the underlying Langevin dynamics are reversible, the finite-time discretization and sampling noise generally yield an estimator $P$ that violates reversibility as seen in Definition~\ref{def:reversibility}.  
Reversibility is typically restored by post-processing, e.g., through reversibilization projections~\cite{Choi_Wolfer} or constrained optimization~\cite{durastante2025riemannianoptimizationapproachfinding,MR3338930}.

\subsubsection{An example of Markov chains constructed from the counting matrix}

As a concrete example of the Langevin dynamics introduced in~\eqref{eq:sde}--\eqref{eq:Euler--Maruyama}, 
we consider the torsional motion of the butane molecule, modeled by the potential
\[
U(x) = a + b \cos(x) + c \cos^2(x) + d \cos^3(x),
\]
which provides an effective approximation of the butane dihedral energy profile~\cite{MR3338930}. 
Here, \(x\) denotes the central dihedral angle, and the parameters are set to 
\(a = 2.0567\), \(b = -4.0567\), \(c = 0.3133\), and \(d = 6.4267\). 
Since \(U(x)\) is periodic with period \(2\pi\), we restrict the state space to the interval \([0,2\pi]\), 
which we partition into \(B=30\) disjoint subsets.
The Langevin dynamics are then integrated using the Euler--Maruyama scheme~\eqref{eq:Euler--Maruyama}, 
with step size \(\Delta t = 10^{-3}\), noise intensity \(\sigma = 1\), 
and total number of steps \(N = 5 \times 10^8\). 
The resulting discrete trajectory is subsequently coarse-grained by the partition, 
yielding the state sequence used for the construction of the Markov model as described in the previous Section~\ref{sec:transition_counts}.

We generate a transition matrix $P$ from the count matrix corresponding to this configuration and apply the proposed optimization strategy to recover a reversible approximation. The resulting matrix $P$ has no transient states and consists of a single ergodic class, so the reversible approximation can be computed directly on the full state space. We then compute the nearest reversible matrix according to our formulations employing the \lstinline{'gurobi'} QP algorithm and summarize the results in Fig.~\ref{fig:butane_example}.
\begin{figure}[htb]
    \centering
    \definecolor{mycolor1}{rgb}{0.12941,0.12941,0.12941}%
\begin{tikzpicture}

\begin{axis}[%
width=0.182\columnwidth,
height=0.182\columnwidth,
at={(0\columnwidth,0\columnwidth)},
scale only axis,
xmin=0,
xmax=31,
xlabel style={font=\color{mycolor1}},
xlabel={nz = 90},
y dir=reverse,
ymin=0,
ymax=31,
axis background/.style={fill=white},
title style={font=\bfseries\color{mycolor1}\footnotesize},
title={$P$}
]
\addplot [color=black, only marks, mark size=0.6pt, mark=*, mark options={solid, black}, forget plot]
  table[row sep=crcr]{%
1	1\\
1	2\\
1	30\\
2	1\\
2	2\\
2	3\\
3	2\\
3	3\\
3	4\\
4	3\\
4	4\\
4	5\\
5	4\\
5	5\\
5	6\\
6	5\\
6	6\\
6	7\\
7	6\\
7	7\\
7	8\\
8	7\\
8	8\\
8	9\\
9	8\\
9	9\\
9	10\\
10	9\\
10	10\\
10	11\\
11	10\\
11	11\\
11	12\\
12	11\\
12	12\\
12	13\\
13	12\\
13	13\\
13	14\\
14	13\\
14	14\\
14	15\\
15	14\\
15	15\\
15	16\\
16	15\\
16	16\\
16	17\\
17	16\\
17	17\\
17	18\\
18	17\\
18	18\\
18	19\\
19	18\\
19	19\\
19	20\\
20	19\\
20	20\\
20	21\\
21	20\\
21	21\\
21	22\\
22	21\\
22	22\\
22	23\\
23	22\\
23	23\\
23	24\\
24	23\\
24	24\\
24	25\\
25	24\\
25	25\\
25	26\\
26	25\\
26	26\\
26	27\\
27	26\\
27	27\\
27	28\\
28	27\\
28	28\\
28	29\\
29	28\\
29	29\\
29	30\\
30	1\\
30	29\\
30	30\\
};
\end{axis}

\begin{axis}[%
width=0.182\columnwidth,
height=0.182\columnwidth,
at={(0.239\columnwidth,0\columnwidth)},
scale only axis,
xmin=0,
xmax=31,
xlabel style={font=\color{mycolor1}},
xlabel={nz = 90},
y dir=reverse,
ymin=0,
ymax=31,
axis background/.style={fill=white},
title style={font=\bfseries\color{mycolor1}},
title={$\mathbb{S}(P + P^\top + I)$}
]
\addplot [color=black, only marks, mark size=0.6pt, mark=*, mark options={solid, black}, forget plot]
  table[row sep=crcr]{%
1	1\\
1	2\\
1	30\\
2	1\\
2	2\\
2	3\\
3	2\\
3	3\\
3	4\\
4	3\\
4	4\\
4	5\\
5	4\\
5	5\\
5	6\\
6	5\\
6	6\\
6	7\\
7	6\\
7	7\\
7	8\\
8	7\\
8	8\\
8	9\\
9	8\\
9	9\\
9	10\\
10	9\\
10	10\\
10	11\\
11	10\\
11	11\\
11	12\\
12	11\\
12	12\\
12	13\\
13	12\\
13	13\\
13	14\\
14	13\\
14	14\\
14	15\\
15	14\\
15	15\\
15	16\\
16	15\\
16	16\\
16	17\\
17	16\\
17	17\\
17	18\\
18	17\\
18	18\\
18	19\\
19	18\\
19	19\\
19	20\\
20	19\\
20	20\\
20	21\\
21	20\\
21	21\\
21	22\\
22	21\\
22	22\\
22	23\\
23	22\\
23	23\\
23	24\\
24	23\\
24	24\\
24	25\\
25	24\\
25	25\\
25	26\\
26	25\\
26	26\\
26	27\\
27	26\\
27	27\\
27	28\\
28	27\\
28	28\\
28	29\\
29	28\\
29	29\\
29	30\\
30	1\\
30	29\\
30	30\\
};
\end{axis}

\begin{axis}[%
width=0.182\columnwidth,
height=0.182\columnwidth,
at={(0.479\columnwidth,0\columnwidth)},
scale only axis,
xmin=0,
xmax=31,
xlabel style={font=\color{mycolor1}},
xlabel={nz = 90},
y dir=reverse,
ymin=0,
ymax=31,
axis background/.style={fill=white},
title style={font=\bfseries\color{mycolor1}},
title={$\Delta$}
]
\addplot [color=black, only marks, mark size=0.6pt, mark=*, mark options={solid, black}, forget plot]
  table[row sep=crcr]{%
1	1\\
1	2\\
1	30\\
2	1\\
2	2\\
2	3\\
3	2\\
3	3\\
3	4\\
4	3\\
4	4\\
4	5\\
5	4\\
5	5\\
5	6\\
6	5\\
6	6\\
6	7\\
7	6\\
7	7\\
7	8\\
8	7\\
8	8\\
8	9\\
9	8\\
9	9\\
9	10\\
10	9\\
10	10\\
10	11\\
11	10\\
11	11\\
11	12\\
12	11\\
12	12\\
12	13\\
13	12\\
13	13\\
13	14\\
14	13\\
14	14\\
14	15\\
15	14\\
15	15\\
15	16\\
16	15\\
16	16\\
16	17\\
17	16\\
17	17\\
17	18\\
18	17\\
18	18\\
18	19\\
19	18\\
19	19\\
19	20\\
20	19\\
20	20\\
20	21\\
21	20\\
21	21\\
21	22\\
22	21\\
22	22\\
22	23\\
23	22\\
23	23\\
23	24\\
24	23\\
24	24\\
24	25\\
25	24\\
25	25\\
25	26\\
26	25\\
26	26\\
26	27\\
27	26\\
27	27\\
27	28\\
28	27\\
28	28\\
28	29\\
29	28\\
29	29\\
29	30\\
30	1\\
30	29\\
30	30\\
};
\end{axis}

\begin{axis}[%
width=0.182\columnwidth,
height=0.182\columnwidth,
at={(0.718\columnwidth,0\columnwidth)},
scale only axis,
xmin=0,
xmax=31,
xlabel style={font=\color{mycolor1}},
xlabel={Positive 60 - Negative 30},
y dir=reverse,
ymin=0,
ymax=31,
axis background/.style={fill=white},
axis x line*=bottom,
axis y line*=left
]
\addplot [color=black, only marks, mark=+, mark options={solid, black}, mark size=0.9pt, forget plot]
  table[row sep=crcr]{%
1	1\\
1	30\\
2	1\\
2	2\\
3	2\\
3	3\\
4	3\\
4	4\\
5	4\\
5	5\\
6	5\\
7	6\\
8	7\\
9	8\\
10	9\\
11	10\\
11	11\\
12	11\\
12	12\\
13	12\\
13	13\\
14	13\\
14	14\\
15	14\\
15	15\\
16	15\\
17	16\\
18	17\\
19	18\\
20	19\\
21	20\\
21	21\\
22	21\\
22	22\\
23	22\\
23	23\\
24	23\\
24	24\\
25	24\\
25	25\\
26	25\\
27	26\\
28	27\\
29	28\\
30	29\\
};
\addplot [color=gray, only marks, mark=-, mark options={solid, black}, mark size=0.6pt, forget plot]
  table[row sep=crcr]{%
1	2\\
2	3\\
3	4\\
4	5\\
5	6\\
6	6\\
6	7\\
7	7\\
7	8\\
8	8\\
8	9\\
9	9\\
9	10\\
10	10\\
10	11\\
11	12\\
12	13\\
13	14\\
14	15\\
15	16\\
16	16\\
16	17\\
17	17\\
17	18\\
18	18\\
18	19\\
19	19\\
19	20\\
20	20\\
20	21\\
21	22\\
22	23\\
23	24\\
24	25\\
25	26\\
26	26\\
26	27\\
27	27\\
27	28\\
28	28\\
28	29\\
29	29\\
29	30\\
30	1\\
30	30\\
};
\end{axis}

\end{tikzpicture}%
    \caption{Reversibilization of the transition matrix $P$ obtained from the count matrix of the butane configuration. 
    The four panels, from left to right, show: (i) the adjacency matrix of the original Markov chain $P$; 
    (ii) the admissible modification pattern $\mathbb{S}(P + P^\top + I)$; 
    (iii) the sparsity pattern of the perturbation $\Delta$ obtained from the Nearest Sparse Reversible problem; (iv) the sign of the modification.}
    \label{fig:butane_example}
\end{figure}
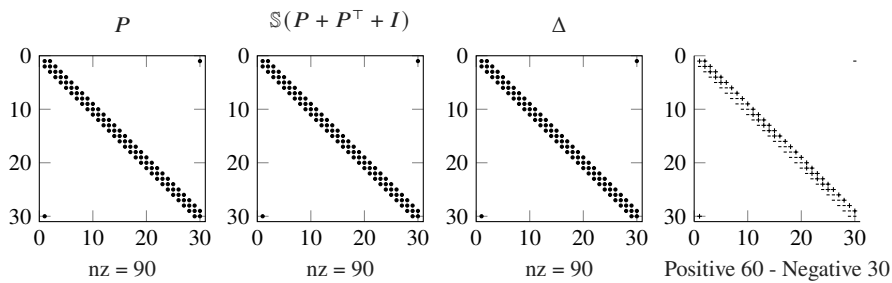
The norm of the modification is $\|\Delta\|_F \approx 0.021858$, which is inferior to the norm of the perturbation with respect to the Metropolis--Hastings reversibilization in~\eqref{eq:reversibilization}-\eqref{eq:reversibilization_metropolis} which sits at a distance $\approx 0.057561$. Moreover, the constraint residuals satisfy $\|(P+\Delta)\mathbf{1}-\mathbf{1}\|_\infty \approx 2.2204\times 10^{-16}$, $\|D_{\boldsymbol{\pi}}(P+\Delta)-(P+\Delta)^\top D_{\boldsymbol{\pi}}\|_\infty \approx 2.1684\times 10^{-18}$, and $\|\boldsymbol{\pi}^\top(P+\Delta)-\boldsymbol{\pi}^\top\|_\infty \approx 5.5511\times 10^{-17}$. These values indicate that the computed solution satisfies all constraints to machine precision. %
The modification redistributes the weights within the existing symmetric sparsity pattern. In particular, most diagonal entries and the lower first off-diagonal entries $(i,i-1)$ are increased (marked by ``$+$''), whereas the upper first off-diagonal entries $(i,i+1)$ are predominantly decreased (marked by ``$-$''). In addition, the corner entry $(1,30)$ is increased, while its symmetric counterpart $(30,1)$ is decreased. Overall, the sparsity pattern remains unchanged, but the mass is shifted from the upper diagonal toward the main and lower diagonals.

\subsubsection{A more elaborate approach using a Markov State Model}

We consider here a more elaborate construction than the simple estimate of the Markov chain obtained from the count matrix approach described in Section~\ref{sec:transition_counts}. As a test system, we considered a dataset~\cite{McGibbon2014} for Fs-peptide, a synthetic 21-residue peptide (Ace-A\textsubscript{5}-(AAARA)\textsubscript{3}-A-Nme) that is widely used as a minimal model for studying $\alpha$-helical folding\footnote{Each trajectory in the simulated data for this protein system is 500 \si{\nano\second} in length, and saved at a 50 \si{\pico\second} time interval via 14 \si{\micro\second} aggegrate sampling. The simulations were performed using the AMBER99SB-ILDN force field with GBSA-OBC implicit solvent at 300\si{\kelvin}, starting from randomly
sampled conformations from an initial 400\si{\kelvin} unfolding simulation. The simulation in the dataset have been run with the OpenMM (\href{https://openmm.org/}{openmm.org}) software.}. Its relatively small size, combined with rich conformational dynamics on the nanosecond to microsecond timescale, makes it an ideal benchmark for evaluating kinetic models. The construction described below follows closely the Fs-peptide example provided in the MSMBuilder library\footnote{See the example \href{http://msmbuilder.org/3.8.0/examples/Fs-Peptide-in-RAM.html}{http://msmbuilder.org/3.8.0/examples/Fs-Peptide-in-RAM.html}.}, which we reproduce here in order to illustrate how we arrive at a non-reversible Markov chain from the elaboration of a dynamical model.

The procedure consists of \emph{featurization}, \emph{dimensionality reduction}, \emph{clustering}, and \emph{estimation of the transition matrix}. First, the raw molecular dynamics trajectories, given as atomic coordinates $\{\mathbf{r}_t\}_{t=0}^T $, were transformed into an internal coordinate representation---i.e., a \emph{featurizer} is used which transforms Cartesian coordinates into a most suitable coordinate systems. In particular, we extracted the set of backbone dihedral angles $(\phi, \psi)$, producing a sequence of feature vectors $\mathbf{x}_t \in \mathbb{R}^d$, where $d$ is the number of features; this reduces the space from the original $264 \times 3$-dimensional space to $d = 84$ dimensions. Dihedral angles are then centered and scaled our by their respective interquartile ranges to reduce bias in the data.

Next, we applied time-structure independent component analysis (tICA)~\cite{doi:10.1021/acs.jctc.7b00182,PhysRevLett.72.3634} to identify the slow collective motions. Intuitively, tICA can be understood as a time-aware version of principal component analysis (PCA). While PCA finds directions of maximum variance, tICA finds directions along which the system evolves as slowly as possible, i.e., directions that maximize time-lagged correlations. In matrix terms, let $C_0 = \mathbb{E}[\mathbf{x}_t \mathbf{x}_t^\top]$ be the covariance matrix and $C_\tau = \mathbb{E}[\mathbf{x}_t \mathbf{x}_{t+\tau}^\top]$ the time-lagged covariance at lag time $\tau$. The tICA components are obtained by solving the generalized eigenvalue problem
\[
C_\tau \mathbf{w} = \lambda C_0 \mathbf{w},
\]
where $\lambda$ measures the autocorrelation of the signal projected along direction $\mathbf{w}$. The dominant eigenvectors $\{\mathbf{w}_i\}$ define linear projections $\mathbf{y}_t = W^\top \mathbf{x}_t$ that emphasize the slowest modes of the peptide’s dynamics. In our application this reduces the number of dimensions from $84$ to $4$; Fig.~\ref{fig:molecular_dynamic_examples_a} shows an histogram of the reduced data projected along the two slowest degrees of freedom as identified by the tICA.
\begin{figure}[htbp]
    \centering
    \subfloat[\label{fig:molecular_dynamic_examples_a}]{\includegraphics[width=0.45\columnwidth]{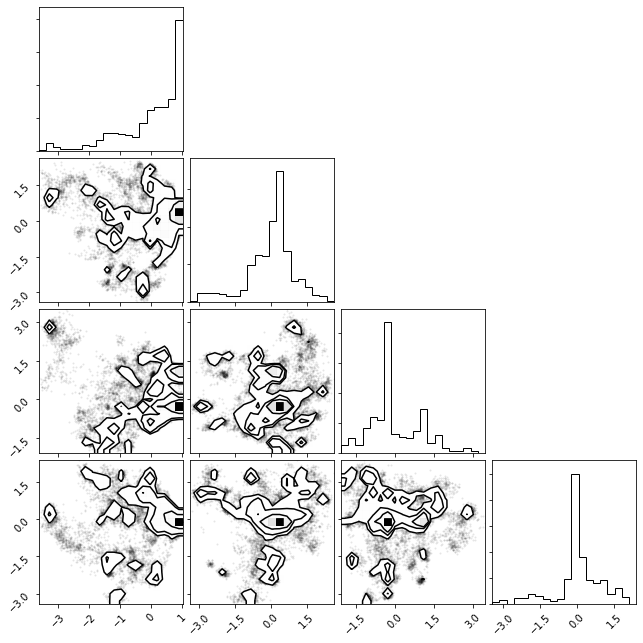}}\hfil
    \subfloat[\label{fig:molecular_dynamic_examples_b}]{\includegraphics[width=0.45\columnwidth]{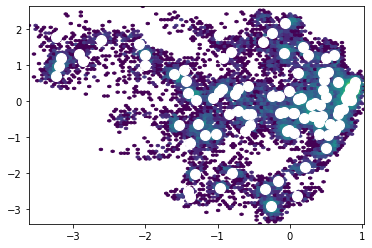}}
    \caption{Illustration of the MSM construction for the Fs-peptide system. Panel~\ref{fig:molecular_dynamic_examples_a} reports the projection of the trajectories onto the first two tICA components, and is followed by panel~\ref{fig:molecular_dynamic_examples_b} which gives discretization into clusters/microstates obtained by using MiniBatch $k$-Means~\cite{10.1145/1772690.1772862}---cluster centers are represtend as large white dots. These clusters serve as the discrete states on which the Markov State Model is estimated.}
    \label{fig:molecular_dynamic_examples}
\end{figure}
The reduced coordinates $\{\mathbf{y}_t\}$ were then discretized by clustering with MiniBatch $k$-Means~\cite{10.1145/1772690.1772862}. Each MD frame was assigned to the nearest cluster center, producing a sequence of discrete microstate labels $s_t \in \{1,\dots,K\}$, where $K$ is the number of clusters as depicted in Fig.~\ref{fig:molecular_dynamic_examples_b}. From these microstate trajectories we estimated the transition probability matrix $P(\tau)$ at lag time $\tau$ by counting transitions:
\begin{equation}\label{eq:t_peptide}
    P_{ij}(\tau) = \frac{C_{ij}(\tau)}{\sum_j C_{ij}(\tau)}, \qquad P \mathbf{1} = \mathbf{1},
\end{equation}
where $C_{ij}(\tau)$ is the number of observed transitions from state $i$ to $j$ separated by $\tau$ time units.

We now apply the proposed optimization strategy to compute the nearest reversible transition matrix to the estimate $P$ obtained from~\eqref{eq:t_peptide} using the procedure described above. The estimated matrix $P$ consists of a single ergodic class and has no transient states; consequently, the optimization is performed over the full state space. The corresponding results are shown in Fig.~\ref{fig:reversiblized_peptide}.

\begin{figure}[htbp]
    \centering
    \input{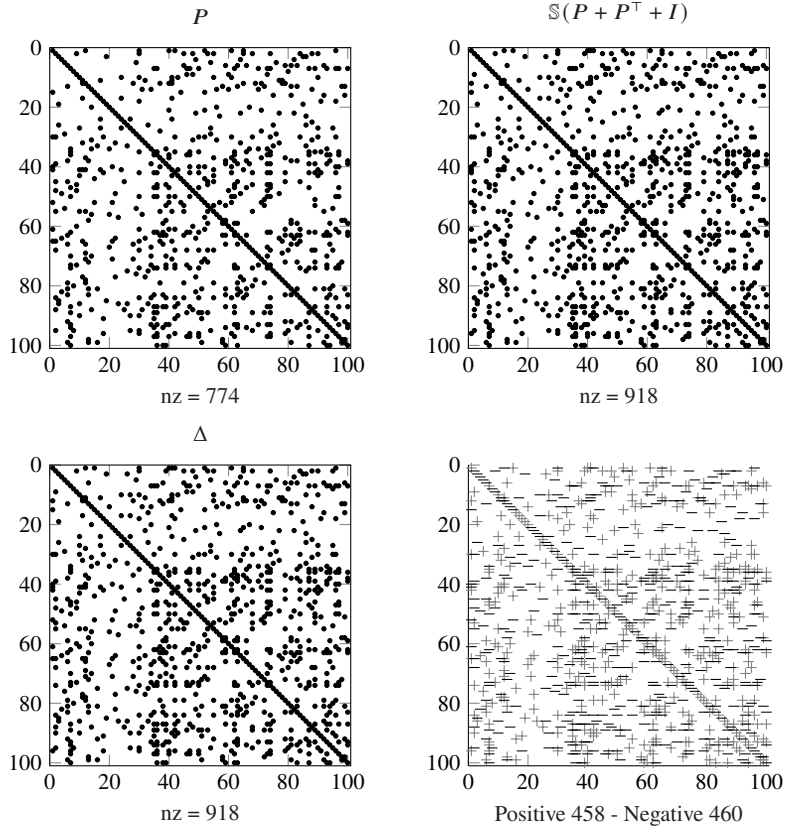}
    \caption{Reversibilization of the transition matrix $P$ from~\eqref{eq:t_peptide}. The four panels, from left to right, show: (i) the sparsity pattern of the original matrix $P$, which is not symmetric; (ii) the admissible modification pattern $\mathbb{S}(P + P^\top + I)$; (iii) the sparsity pattern of the perturbation $\Delta$; (iv) sign of the obtained $\Delta$ matrix.}
    \label{fig:reversiblized_peptide}
\end{figure}

The perturbation modifies all entries permitted by the admissible pattern, resulting in $\operatorname{nnz}(\Delta) = 918$. The Frobenius norm of $\Delta$ is approximately $0.134822$, which is approximately five times smaller than the distance $\approx 0.653852$ achieved by the Metropolis-Hastings reversibilization~\eqref{eq:reversibilization}-\eqref{eq:reversibilization_metropolis}. The optimization was carried out using the \lstinline{'gurobi'} solver option, with solution times of $0.0036$~\si{\second}. The constraint residuals satisfy $\|(P+\Delta)\mathbf{1}-\mathbf{1}\|_\infty \approx 4.440892\times 10^{-16}$, $\|D_{\boldsymbol{\pi}}(P+\Delta)-(P+\Delta)^\top D_{\boldsymbol{\pi}}\|_\infty \approx 2.775558\times 10^{-17}$, and $\|\boldsymbol{\pi}^\top(P+\Delta)-\boldsymbol{\pi}^\top\|_\infty \approx 1.396265\times 10^{-17}$.

\section{Conclusions}\label{sec:conclusions}

In this work, we have developed a comprehensive optimization framework for computing the nearest reversible sparse Markov chain to a given non-reversible matrix. The reversibility constraint is expressed through the symmetry condition with respect to the stationary distribution, and we have embedded this constraint into the optimization problem. A key theoretical contribution is the proof that the underlying optimization problem is strongly convex, ensuring the existence and uniqueness of the solution and, crucially, its independence from the choice of optimization algorithm. This theoretical property provides robustness and consistency across different computational approaches. Through comparative analysis, we evaluated two state-of-the-art optimization algorithms---\texttt{quadprog} and \texttt{gurobi}---demonstrating their practical viability for both small and moderate-scale problems. Our experiments reveal that \texttt{gurobi} provides superior computational performance, achieving solve times approximately 3--4 times faster than \texttt{guadprog}, while all methods satisfy the reversibility and stationarity constraints to machine precision. The subquadratic scaling of computational cost with problem dimension underscores the scalability of the approach.

Extensive numerical experiments on synthetic sparse matrices and real-world examples from molecular dynamics simulations validate the proposed methodology. The results demonstrate that reversibility introduces a systematic and quantifiable increase in sparsity, an intrinsic consequence of the coupling constraints between forward and backward transitions. This fundamental trade-off is independent of the specific problem instance and algorithmic choice, suggesting that the sparsity-reversibility relationship reflects a deep structural property of Markov chains. Furthermore, the framework successfully processes transition matrices obtained from empirical data (transition counts) and recovers reversible matrices with high constraint satisfaction. We plan to evaluate the usage of regularized interior point methods from~\cite{MR4594481,MR4784765}, to further accelerate the solution and increase the robustness of the method{, and to investigate the case where we also add the topology of the underlying graph to the variables to be optimized; see again Remark~\ref{rmk:miqp_formulation}.}

\begin{acknowledgement}
This work was partially supported by the Italian Ministry of University and Research (MUR) through the PRIN 2022 ``Low-rank Structures and Numerical Methods in Matrix and Tensor Computations and their Application'' code: 20227PCCKZ MUR D.D. financing decree n. 104 of February 2nd, 2022 (CUP I53D23002280006) and through the PRIN 2022 ``MOLE: Manifold constrained Optimization and LEarning'', code: 2022ZK5ME7 MUR D.D. financing decree n. 20428 of November 6th, 2024 (CUP B53C24006410006). The last three authors are member of the INdAM GNCS and acknowledge the MUR Excellence Department Project awarded to the Department of Mathematics, University of Pisa, CUP I57G22000700001. Furthermore, during part
of the preparation of this work, MG was affiliated with the Department of Mathematics, University of Pisa, Pisa (Italy).

We are grateful to the reviewers for their constructive feedback, which allowed us to improve the presentation of our findings.
\end{acknowledgement}

\ethics{Competing Interests}{The authors have no conflicts of interest to declare that are relevant to the content of this chapter.}

\bibliographystyle{spmpsci}
\bibliography{biblio}

\end{document}